\documentclass{article}
\usepackage{authblk}
\newcommand{\Adresses}{{
\bigskip
\footnotesize

Octave Curmi, AIX MARSEILLE UNIVERSIT\'E, CNRS, CENTRALE MARSEILLE, I2M, FRANCE.

\emph{E-mail address: }octave.curmi@univ-amu.fr
}}
\usepackage[toc,page]{appendix}
\usepackage{amsmath}
\usepackage{amsthm}
\usepackage[pagebackref]{hyperref}
  \hypersetup{
     colorlinks=true,
     linkcolor=blue,
    citecolor=blue
   }
\usepackage{url}
\usepackage[normalem]{ulem}
\usepackage[refpage]{nomencl}
\makenomenclature
\usepackage{tikz}
\usetikzlibrary{matrix}
\usepackage[scale=0.7]{geometry}
\usepackage{amssymb}
\usepackage{accents}
\usepackage[utf8]{inputenc}
\usepackage{float}
\usepackage{etoolbox}
\usepackage[english]{babel}
\usepackage{blindtext}
\usepackage[T1]{fontenc}
\usepackage{mathtools}
\usepackage[nottoc]{tocbibind}
\usepackage{enumitem}
\usepackage{faktor}
\usepackage{xfrac}
\usepackage{calrsfs}
\usepackage{makeidx}
\usepackage{xcolor}
\DeclareMathAlphabet{\pazocal}{OMS}{zplm}{m}{n}

\let\mathcal\pazocal
\let\pazocal\temp

\newcommand\C{\mathbb{C}}
\newcommand\R{\mathbb{R}}
\newcommand\Z{\mathbb{Z}}
\newcommand\N{\mathbb{N}}

\newcommand\Sp{\mathbb{S}}

\newtheorem*{theo}{Theorem}
\newtheorem{theorem}{Theorem}[section]
\newtheorem{lemma}[theorem]{Lemma}

\newtheorem{prop}[theorem]{Proposition}
\newtheorem{cor}[theorem]{Corollary}
\newtheorem{defn}[theorem]{Definition}
\newtheorem{rk}[theorem]{Remark}
\newtheorem{ex}[theorem]{Example}
\newtheorem{nota}[theorem]{Notation}
\newtheorem{thm}[theorem]{Theorem}

\newtheorem{discussion}[theorem]{Discussion}

\newtheorem{conditionk}[theorem]{Condition}

\newcommand\Dt{\mathbb{D}}
\newcommand\Df{\mathbb{D}_{f}}

\newcommand\Dfex{{\mathbb{D}_{f,ex}}}

\newcommand\Do{\mathbb{D}_{0}}
\newcommand\Ex{E}

\newcommand\Vf{V(f)}
\newcommand\Vg{V(g)}
\newcommand\Vft{\widetilde{V(f)}}
\newcommand\Vgt{\widetilde{V(g)}}
\newcommand\Dg{\widetilde{V(g)}}
\newcommand\V{V_{f\cdot g}} 
\newcommand\Ft{\widetilde{F}} 
\newcommand\dF{\partial \F}
\newcommand\F{F}

\newcommand\ft{{f_{\Xt}}}
\newcommand\gt{{g_{\Xt}}}
\newcommand\gts{{g_{_{\Skt}}}}
\newcommand\gk{{g^\C}}

\newcommand\gti{\tilde{g}}
\newcommand\rx{r_X}

\newcommand\tr{r_{_{\pazocal{S}}}}
\newcommand\norm{N_{_{\pazocal{S}}}}
\newcommand\kp{\kappa^\oplus}
\newcommand\km{\kappa^\ominus}
\newcommand\K{K}

\newcommand\resol{\pi}
\newcommand\dist{\rho} 
\newcommand\NK{{\left(\norm\circ \K\right)}}

\newcommand\X{X}
\newcommand\Xt{\tilde{X}}
\newcommand\Li{\Sp} 

\newcommand\dk{{\delta_k}}
\newcommand\ek{{\varepsilon_k}}

\newcommand\Fs{F^\square}
\newcommand\Sks{\Sk^\square}

\newcommand\Courbes{\pazocal{C}} 
\newcommand\Courbesn{\norm^*(\Courbes)} 

\newcommand\Courbesk{\left(\norm\circ \K\right)^*(\Courbes)} 
\newcommand\Courbeskt{\left(\norm\circ \K\right)^*(\Courbest)}
\newcommand\Ck{\left(\norm\circ \K\right)^*(C)} 
\newcommand\Cuk{\left(\norm\circ \K\right)^*(C_1)}
\newcommand\Cdk{\left(\norm\circ \K\right)^*(C_2)}
\newcommand\Courbest{\pazocal{C}_{tot}} 

\newcommand*{\str}[1][E]{{#1}^{str}} 

\newcommand\Upt{\overline{U_p}} 
\newcommand\Uqt{\overline{U_q}}
\newcommand\Upit{\overline{U_{p_i}}}
\newcommand\Upk{U_p^\C} 
\newcommand\Uqk{U_q^\C}
\newcommand\Upik{U_{p_i}^\C}

\newcommand{\ba}[1][E]{\overline{#1}}
\newcommand\pt{{\ba[p]}}
\newcommand\Ct{\ba[C]}
\newcommand\Cpt{\ba[C']}
\newcommand\Ut{\ba[U]}

\newcommand{\ti}[1][E]{\widetilde{#1}}

\newcommand\Cti{\ti[C]}
\newcommand\Cpti{\ti[C']}
\newcommand\Uti{\ti[U]}

\newcommand\VG{{\pazocal{V}(\Gamma)}}
\newcommand\EG{{\pazocal{E}(\Gamma)}}

\newcommand\PV{{\pazocal{V}}}
\newcommand\PE{{\pazocal{E}}}

\newcommand*{\ve}[1][E]{v_{#1}} 

\newcommand*{\Gra}[1][E]{\Gamma(#1)} 
\newcommand\Gn{\Gra[\Courbesn]}

\newcommand\Gc{\Gra[\Courbes]}
\newcommand\Gct{\Gra[\Courbest]}
\newcommand\Gk{\Gra[\Courbesk]}
\newcommand\Gkt{\Gra[\Courbeskt]}
\newcommand\Gresol{\Gra[E]}

\newcommand*{\Grad}[1][E]{\overset{\star}{\Gamma}\left(#1\right)} 

\newcommand\Gcd{\Grad[\Courbes]}
\newcommand\Gkd{\Grad[\Courbesk]}
\newcommand\Gresold{\Grad[E]}

\newcommand\Gcdtpart{\Grad[\Courbest]}

\newcommand\Gcdt{\Grad[\Courbest]}
\newcommand\Gkdt{\Grad[\Courbeskt]}
\newcommand\Gresoldc{\Grad[E]}
\newcommand\Gmult{\Gamma^\mu (E_{tot})}

\newcommand*{\Gdp}[2]{\Gamma_{#2}({#1})}

\newcommand\plus{\oplus}
\newcommand\minus{\ominus}

\newcommand\Sk{\pazocal{S}_k}
\newcommand\Sb{{\widetilde{\pazocal{S}}}}
\newcommand\Skt{{\widetilde{\pazocal{S}_k}}}
\newcommand\St{{\overline{\pazocal{S}}}}
\newcommand\Sn{{\Skt^{N}}}
\newcommand\Zk{Z_k}

\newcommand\Crit{\operatorname*{Crit}}

\newcommand{\sm}[1]{{#1}^0}
\newcommand{\si}[1]{\Sigma #1}

\author{Octave Curmi}
\title{Topology of smoothings of non-isolated singularities of complex surfaces}
\date{}
\begin{document}

\maketitle

\begin{abstract}
We prove that the boundaries of the Milnor fibers of smoothings of non-isolated reduced complex surface singularities are graph manifolds. Moreover, we give a method, inspired by previous work of Némethi and Szilard, to compute associated plumbing graphs. 
\end{abstract}

\section{Introduction}

The study of Milnor fibers of complex-analytic functions, which began in the second half of the 20th century, gave rise to a rich interaction between algebra and topology. 

John Milnor provided in \cite{Mil56} the first examples of exotic spheres, that is, smooth manifolds homeomorphic but not diffeomorphic to the $n$-sphere, in the case $n=7$. 

Meanwhile, a first relation between the topology of varieties and complex analytic regularity was discovered by Mumford, in \cite{Mum61}, where he established that a $2$-dimensional normal complex surface which is a topological manifold is nonsingular. Exploring the possibility of extending such results to higher dimensions, Brieskorn proved in \cite{Bri66} that it is no longer the case, producing many examples of singularities which are topological manifolds, namely all singularities of the form $$V=\{z_1^2+\cdots +z_k^2-z_0^3=0\}\subset \C^{k+1} , ~k\text{ odd and }k \geq 3.$$

The \textbf{link} $V\cap \Sp_\varepsilon$ of such singularities is always a knotted sphere, and in some cases it is an exotic sphere, as was shown by Hirzebruch in the case $k=5$, see \cite[p 46-48]{Bri00} and \cite{Hirz95}. See also \cite[Chapter 9]{Mil68}, and \cite[Section 1]{Sea19}.

This potential fabric of exotic spheres led Milnor to study further the topology of hypersurface singularities, and eventually to write his famous 1968 book \cite{Mil68}, aimed at the study of isolated singularities of hypersurfaces of $\C^n$, $V(f)=\{f=0\}$. In this work he introduced two equivalent fibrations, using respectively the levels of $f/|f|$ on spheres centered at the critical points of $f$ and the levels of $f$ in $\C^n$. 

The second of these two fibrations has been extended to more general contexts by Lê, see \cite{Le77}, and is known as the \textbf{Milnor-Lê fibration}. However, it may produce singular generic fibers, due to the singularities of the ambient space. Hamm, in \cite{Ham71}, provided a setting in which the Milnor-Lê fibration is actually a \textbf{smoothing} of $V(f)$, that is, a way to put $V(f)$ in a \textbf{flat} family of analytic spaces, where the generic space is smooth. Namely, if $(X,0)$ is a germ of equidimensional complex analytic space, and $f$ is any holomorphic function on $(X,0)$ such that $V(f)\supset \Sigma X$, then the function $f$ provides a smoothing of the singularity $(V(f),0)$, where $\Sigma X$ denotes the singular locus of $X$.

\bigskip

This work is dedicated to the study of smoothings of complex analytic surface singularities. We start with a germ of complex surface $(V,p)$, and, when there exists a smoothing $f \colon (X,0)\rightarrow (\C,0)$, we study the generic fiber of this smoothing.

A singularity of surface may admit different smoothings, or even none. The study of the topology of the fibers of the smoothings of a given singularity of complex surface is very hard, even for isolated ones, and there is only a few types of singularities where a description of the full fiber is known. It is the case for the Kleinean sigularities A, D, E, where the Milnor fiber is unique and diffeomorphic to the minimal resolution (see Brieskorn, \cite{Bri66.2}), as well as for the singularities of normal toric surfaces, with a description by surgery (see Lisca, \cite{Lis08} and Némethi \& Popescu-Pampu, \cite{NemPop10}), and for sandwich singularities (De Jong \& Van Straten, \cite{DejVan98}). As for nonisolated singularities, the only known case is that of hypersurface singularities of the form $\{f(x,y)+z\cdot g(x,y)=0\}$, see Sigur\dh sson, \cite{Sig16}. One can also refer to the survey \cite{Gre19} for more about the topology of smoothings and deformations of singularities.

On the other hand, the study of the \textbf{boundary} of the Milnor fiber has been a very active area of research in the last decades.

The boundary of the fiber of a smoothing of an isolated singularity is unique, and equal to the link of the singularity. In today's words, Mumford proved in \cite{Mum61} that the link of any isolated singularity of complex surface is a \textbf{graph manifold}, that is, a manifold describable using a decorated graph whose vertices represent fibrations in $\Sp^1$ over compact surfaces. It is Waldhausen, in \cite{Wal67}, that later introduced this vocabulary and began studying this class of varieties in itself. Furthermore, since the work of Grauert (\cite{Gra62}), one knows exactly which graph manifolds appear as links of isolated singularities of complex surfaces. However, this strong point is tempered by the fact that one still does not know, for example, which of these manifolds appear as links of singularities of \textbf{hypersurfaces} of $\C^3$.

Still, one would dream of having an analogous result for boundaries of Milnor fibers associated to non isolated singularities. The first steps towards the comprehension of the topology of these manifolds were made by Randell \cite{Ran77}, then Siersma in \cite{Sie91}, \cite{Sie01}, who computed the homology of the boundary $\partial F$ of the Milnor fiber in certain cases, and characterized the cases in which $\partial F$ is a rational homology sphere. 

Concerning the general topology of this manifold, a series of results were aimed at proving that the boundary of the Milnor fibers associated to a non isolated singularity are, again, graph manifolds. A first proof of this result was sketched in 2003 by Michel and Pichon (\cite{MicPic03}) and corrected in \cite{MicPic04}, for the case where the total space of the smoothing is smooth and the function $f$ is reduced. In 2016, they published the details of their proof in \cite{MicPic16}. In collaboration with Weber, they provided explicit plumbing graphs for several classes of singularities: \textbf{Hirzebruch surface singularities} in \cite{MicPicWeb07}, and the so-called \textbf{suspensions} ($f=g(x,y)+z^n$) in \cite{MicPicWeb09}. Fern\'{a}ndez de Bobadilla and Menegon Neto, in \cite{FerMen}, proved it in the context of smoothings of non-isolated and not necessarily reduced singularities whose total space has an isolated singularity, for a function of the form $f\cdot\overline{g}$, with $f$ and $g$ holomorphic. But none of these approaches was constructive. By contrast, Némethi and Szil\'ard gave a constructive proof for the case of reduced holomorphic functions $f \colon (\mathbb{C}^3, 0) \to (\mathbb{C}, 0)$ in their 2012 book \cite{NemSzi12}. Namely, they provided an algorithm to compute the boundary of the Milnor fiber of the function $f$.

However, in general, the total space of a smoothing of a non-isolated singularity of surface has no reason to be smooth, or even with isolated singularity. For example, as soon as the 
generic singularity $(V,q)$ along an irreducible component of the singular locus of $V$ can not be holomorphically embedded in $\C^3$, this means that it can not be seen as a hypersurface of a variety with isolated singularity. It is therefore essential to generalize this result concerning the topology of the boundary of the Milnor fiber to more general spaces.

In this work, we extend the strategy developed in \cite{NemSzi12} by Némethi and Szil\'ard and we prove the following, see Theorem \ref{thm main}:
\begin{theo}\label{th main intro} Let $(\X,0)$ be a germ of $3$-dimensional complex analytic variety, and $f \colon (\X,0)\rightarrow (\mathbb{C},0)$ a germ of reduced holomorphic function on $(\X,0)$, such that $\Vf$ contains the singular locus  $\si[\X]$ of $X$. Then the boundary of the Milnor fiber of $f$ is an oriented $3$-manifold, which can be represented by an orientable plumbing graph.
\end{theo}

We prove this theorem in a constructive way, by adapting to our more general context the method of Némethi and Szil\'ard. As was the case in their book, our proof gives rise  to a method for the computation of $\partial F$. It may therefore represent a step in the obtention of a characterization of the graph manifolds which appear as boundaries of Milnor fibers. Let us insist on the fact that the singular locus of the ambient germ of variety $(X,0)$ induces a series of complications in the proof, implying in particular the need for more data than in \cite{NemSzi12} in order to perform the computations (see Section \ref{section configuration complexification}), as well as the introduction of the notion of companion (see Definition \ref{companion}), which is made to obtain the analogue of ICIS on germs which may not be complete intersections. We also treated the questions of orientations in a more global way, which we hope to be more natural. An important common point with the situation of \cite{NemSzi12}, and the main idea of their proof, is the introduction of a real analytic germ $(\Sk,0)$ of dimension $4$ with isolated singularity having the same boundary as the Milnor fiber, therefore actually realizing this manifold as the link of a singularity. The fact that this variety is real analytic, and not complex-analytic, imposes the use of $\minus$ decorations for some edges of the graph, which never happens in the case of isolated complex surface singularities.
\bigskip

I would like to thank Patrick Popescu-Pampu for his longterm support and comments around this work, as well as Javier Fern\'andez de Bobadilla, Michel Boileau, Andr\'as Némethi, Anne Pichon and David Trotman for helpful discussions and references.

I thank also the AGT team in Aix-Marseille Université and the ANR LISA for allowing me to attend many enriching events.
\bigskip

Each section of the paper starts with a brief description of its content.
\tableofcontents
\section{Background material}

\subsection{Milnor fibration on a singular variety}
In the sequel, a {\em variety} will mean a reduced and equidimensional complex analytic space. 

Let $(X,0)$ be a germ of complex variety of dimension $3$, and let $f \colon (X,0)\rightarrow (\mathbb{C},0)$ be a germ of holomorphic function on $(\X,0)$ assumed not to divide zero, that is, not to vanish on any irreducible component of $(X,0)$. The function $f$ defines a germ $(\Vf,0)$ of hypersurface on $(\X,0)$, where $\Vf:=\{x\in \X, f(x)=0\}$. Denote by $\Sigma \Vf$ the singular locus of $\Vf$, and by $\Sigma \X$ the one of $X$. 

In the sequel, if $(X,0)$ is a germ, $X$ will denote a representative of this germ.

\begin{defn}
We say that a real analytic function $\rho \colon X \rightarrow \R_+$ \textbf{defines} $0$ in $X$ if $0$ is isolated in $\rho^{-1}(0)$, i.e. if there is another representative $X'\subset X$ of $(X,0)$ such that $\rho_{|X'}^{-1}(0)=\{0\}$.
\end{defn}

\begin{theorem}
\emph{(H. Hamm, \cite[Satz 1.6]{Ham71}, Lê \cite[Theorem 1.1]{Le77})}\label{thm fibration}

\noindent Given a real analytic function $\rho$ defining $0$ in $\X$, and $\varepsilon>0$, denote $\X_{\varepsilon}:=\X\cap \rho^{-1}([0,\varepsilon))$, and $\Li_\varepsilon:=\X\cap \{\rho=\varepsilon\}$. Let $f\colon (X,0)\rightarrow (\C,0)$ be a germ of holomorphic function, such that $\X\setminus \Vf$ is smooth. Then there exists $\varepsilon_0>0$, such that $\forall ~ \, 0<\varepsilon \leqslant \varepsilon_0, \exists ~ \delta_\varepsilon > 0$ such that $\forall ~  \,0<\delta\leqslant \delta_\varepsilon$, the following two maps are diffeomorphic smooth fibrations:

\begin{itemize}
\item $\frac{f}{|f|}\colon \Li_\varepsilon \setminus \Vf \rightarrow \Sp^1$

\item $f\colon \partial\left(\left\{|f|=\delta\right\} \cap \X_\varepsilon\right)\rightarrow \partial\left(D_\delta\right)$, where $D_\delta$ denotes the closed disc of radius $\delta$ around $0$ in $\C$.
\end{itemize}

\end{theorem}

\begin{defn}
The first of the two fibrations above is referred to as the \textbf{Milnor fibration of $f$}, and the second one is called the \textbf{Milnor-Lê fibration}. 
The closure of the fiber of the Milnor-Lê fibration is called the \textbf{Milnor fiber} of the germ of function $f$.
\end{defn}

\begin{rk}
The Milnor-Lê fibration is also sometimes referred to as the Milnor fibration.

Using transversality arguments, one may show that the diffeomorphism type of the Milnor fiber does not depend on the chosen representative, so we speak about \textbf{the Milnor fiber of the germ of function $f\in \mathcal{O}_{\X,0}$}
\end{rk}

\subsection{Graph manifolds}
We explain here the definition of graph manifold (also called plumbed manifolds) that we refer to. For details, one can consult the foundational articles \cite{Mum61} and \cite{Wal67}, as well as the article \cite{Neu81} for a description of the so-called plumbing calculus and other topological considerations on graph manifolds.

Recall that the Euler number of a fibration in $\Sp^1$ over a surface is an integer $e$ that characterizes entirely the total space of the fibration up to homeomorphism. See \cite[Example 4.6.5]{GomSti} for more details.

\begin{defn}
An \textbf{orientable plumbing graph} is a graph $\Gamma$ with decorations of the following type:
\begin{itemize}
\item Each edge is decorated by a $\plus$ or $\minus$ symbol.
\item Each vertex $v$ is decorated by an Euler number $e_v\in \Z$ and a genus $[g_v]$, $g_v\in \N$. 
\end{itemize}
\end{defn}

\begin{rk}
When representing plumbing graphs, we may omit the decorations $\plus$ and $[0]$.
\end{rk}

\begin{defn}{\bf (Graph manifold)}

If $\Gamma$ is an orientable plumbing graph, define the oriented \textbf{graph manifold} $M_\Gamma$ associated to $\Gamma$ in the following way: for each vertex $v$ of $\Gamma$ decorated by $([g_v],e_v)$, let $M_v$ be the $\Sp^1$-fibration of Euler number $e_v$ over the closed smooth surface $B_v$ of genus $g_v$. Pick an orientation of the base and the fibers so that, taken together, they give the orientation of $M_v$.

Now, let $\lambda_v$ be the number of times the vertex $v$ appears as endpoint of an edge. Remove from $B_v$ disjoint open disks $(D_i)_{1\leqslant i\leqslant\lambda_v}$, consequently removing as many open solid tori from $M_v$. Each $\partial D_i$ is oriented as boundary component of $B_v \setminus \bigsqcup D_i$. Denote by $M_v^b$ the resulting circle bundle with boundary. Denote $\partial M_v^b=\bigsqcup T_i$ a disjoint union of tori.

For every edge between the vertices $v$ and $v'$, glue the manifolds $M_v^b$ and $M_{v'}^b$ in the following way: pick boundary components $T=\partial D\times \Sp^1$ in $M_v^b$, $T'=\partial D' \times \Sp^1$ in $M_{v'}^b$, and glue $T$ and $T'$ according to the matrix $\epsilon\begin{bmatrix}
0 & 1\\
1 & 0 
\end{bmatrix}$, $\epsilon$ being the sign on the edge.

Finally, we use the convention $$M_{\Gamma_1 \bigsqcup \Gamma_2}=M_{\Gamma_1} \# M_{\Gamma_1}$$
\end{defn}

\section{The ambient germ $(X,0)$ can be assumed normal}\label{X normal}
Before proceeding to the heart of the proof, let us reduce our study to the situation where the ambient germ of variety $(\X,0)$ is normal.

\begin{lemma}
Let $(X,0)$ be a germ of complex variety of dimension $3$, and let $N\colon (\overline{X},\overline{x})\rightarrow (X,0)$ be its normalization. Let $f$ be a germ of analytic function on $(X,0)$ such that $V(f) \supseteq \Sigma X$. Then $N$ provides an orientation-preserving isomorphism of complex-analytic manifolds from the Milnor fiber of $\tilde{f} \colon = f \circ N$ defined on the normal germ $(\overline{X},\overline{x})$ to the one of $f$.
\end{lemma}

\begin{proof}
Let $(\overline{\X},\overline{x})=(\X_1,x_1)\bigsqcup \dots \bigsqcup (\X_k,x_k)\xrightarrow{N}(\X,0)$ be the normalization of $(X,0)$. Then $f$ and $r$ induce analytic functions $f_i=f\circ N,~r_i=r\circ N$ on each $\X_i$, and $r_i$ defines $x_i$ in $\X_i$.

\noindent Now, let $\varepsilon>0,~ \delta_\varepsilon>0$ be adapted (in the sense of Theorem \ref{thm fibration}) to every triple $(f_i,r_i,\X_i)$, and to $(f,r,\X)$. In this situation, we now have a disjoint union of Milnor fibers $F_1\bigsqcup \dots \bigsqcup F_k$ in $\X_1\bigsqcup \dots \bigsqcup \X_k$, each associated with the analytic function $f_i$, $F_i=f_i^{-1}(\delta)\cap \X_{i_\varepsilon}$, and $N$ provides a diffeomorphism between $F$, the Milnor fiber of $f$, and $F_1\bigsqcup \dots \bigsqcup F_k$, because $$F=f^{-1}(\delta)\cap \X_\varepsilon \subset X^0,$$ and $X^0$ is itself contained in the set of normal points of $\X$.
\end{proof}

From now on, the ambient variety $X$ will be assumed normal.

\section{The variety $\Sk$}\label{Sk}
The goal of this section is to introduce a 4-dimensional real analytic variety with isolated singularity $(\Sk,0)$, which is built in such a way as to have a link diffeomorphic to the boundary of the Milnor fiber of $f$.

Let us introduce a second function $g$, that will be used as a computational tool.

\begin{defn}
Let $g\colon (\X,0)\rightarrow (\C,0)$ be a reduced holomorphic function. Together, $f$ and $g$ define a germ of morphism $\Phi:=(f,g) : (\X,0) \rightarrow (\C^2,0)$. 

The \textbf{critical locus} of $\Phi$, denoted $C_\Phi$, is the union $$C_\Phi:= \overline{\Crit(\Phi)}\cup \Sigma X$$

where $\Crit(\Phi)$ denotes the set of smooth points of $\X$ where $\Phi$ is not a local submersion.

The \textbf{discriminant locus} of $\Phi$, denoted by $\Delta_\Phi$, is the image of $C_\Phi$ by $\Phi$: $$\Delta_\Phi:= \Phi(C_\Phi).$$
\end{defn}

For our purposes, the function $g$ must be in a generic position relatively to $f$, more precisely we will ask $g$ to be a companion of $f$ in the following sense:

\begin{defn} \label{companion} 
We say that the holomorphic function $g\colon (X,0)\rightarrow (\C,0)$ is a \textbf{companion of $f$} if it satisfies the following conditions:
\begin{enumerate}
\item The surface $\Vg$ is smooth outside the origin. \label{g smooth}
\item $\Vg\cap \Sigma \Vf = \{0\}$. \label{intersection non sauvage}
\item The surfaces $\Vg$ and $\Vf$ intersect transversally outside the origin.\label{intersection transverse}
\item \label{discr} The discriminant locus $\Delta_\Phi$ is an analytic curve.
\end{enumerate}
\end{defn}

\begin{rk}
This definition is analogous to that of ``Isolated Complete Intersection Singularity'' (ICIS) described in \cite[3.1]{NemSzi12}, or in \cite[2.8]{Loo84}, valid only for ambient smooth germs. The denomination ``companion'' is introduced for the first time in our paper. Note that if the couple $(f,g)$ forms an ICIS in $(X,0)=(\C^3,0)$, then $g$ is a companion of $f$ (see \cite[2.8]{Loo84}).
\end{rk}

\begin{lemma}\label{existence companion}
Given a holomorphic function $f\colon (X,0)\rightarrow (\C,0)$ such that $\Sigma \Vf \supset \Sigma X$, there exists a function $g\colon (X,0)\rightarrow (\C,0)$ that is a companion of $f$.

Furthermore, given an embedding $(X,0)\subset (\C^N,0)$ and a function $f$ on $X$, there is a Zariski open dense set $\Omega$ in the space of affine hyperplanes of $\C^N$ at $0$ such that if $\{G=0\}$ belongs to $\Omega$, the restriction $g$ of $G \colon (\C^N,0)\rightarrow (\C,0)$ to $X$ will be a companion of $f$.
\end{lemma}

\begin{proof}
To verify Conditions \ref{g smooth}, \ref{intersection non sauvage} and \ref{intersection transverse}, it is enough to take a hyperplane which contains neither the tangent cone to $\Vf$ nor that to $\Sigma X$ (which is contained in the first one). These finite conditions are therefore realized by a dense subset of hyperplanes through $0$.

Now, the fact that a generic linear form satisfies Condition \ref{discr} follows from the corollary $1$ of section $5$ of \cite{Hir77} and theorem 2.1 of \cite{Le77}.
\end{proof}

\begin{nota}
In the sequel, $g$ will denote a companion of $f$, and $\Phi$ denotes the morphism $(f,g)\colon (X,0)\rightarrow (\C^2_{x,y},0)$.
\end{nota}

\begin{defn}
For any $k \in 2\N^*$, we define the real analytic smooth variety $$\Zk:=\{(x,y)\in \C^2, x=|y|^k\}\subset \C^2.$$
This variety is a real surface which we orient by taking the pullback of the complex orientation of the $y$-axis by the projection on the second coordinate $\colon \left\lbrace\begin{array}{ccc}
\Zk & \rightarrow & \{x=0\}\\
(|y|^k,y) & \mapsto & y
\end{array}\right.$
\end{defn} 


Reminding that $\Delta_\Phi$ is a complex-analytic curve, one can use the Puiseux expansion method relatively to the Newton polyhedron of each of its branches, as briefly exposed in \cite[5.3]{Sha131}. A direct computation then leads to:

\begin{lemma} \label{lemma discriminant}
Let $s_0$ be the slope of the steepest edge of the Newton polygon of the branches $D_2,\cdots,D_n$.

Then for $k \in 2\N^*$, $k>s_0$, there is a neighbourhood $U_k$ of the origin in $\C^2$ on which: $$U_k\cap \left(\Delta_\Phi \setminus V(x)\right) \subset U_k\cap \{|x|>|y|^k\}.$$
\end{lemma}

\begin{rk}
Remark that $s_0$ occurs in fact at first as the steepest slope of all Newton polygons of all branches of $\Delta_\Phi$, and it is indeed equal to the steepest slope of the Newton polyhedron of $\Delta_\Phi$.
\end{rk}

\begin{defn}
A integer $k\in 2\N$, $k>s_0$ is called \textbf{big enough}. A neighbourhood $U_k$ of the origin in $\C^2$ as in Lemma \ref{lemma discriminant} is called \textbf{small enough} for $k$.
\end{defn}

From now on, $k$ will denote a big enough even integer.

\begin{defn}\label{def Sk germ}
For $k\in 2 \N, k>s_0$, define the germ $$(\Sk,0):=\left(\Phi^{-1}(\Zk),0\right)=\left(\{f=|g|^k\},0\right) \subset (\X,0).$$
\end{defn}

We are now going to build actual representatives $f^{-1}(\dk)\cap B_{\ek}$ of $F$ and $\Sk \cap B_{\ek}$ of $\Sk$ adapted to our construction, where $\ek, \dk$ will be positive numbers.

From now on, for reasons of convenience regarding this construction, we fix $k\in 2\N$ big enough and respecting 
Conditions \ref{k large enough}, \ref{cond k bons points.} and \ref{fractions kn1 m2 sont entiers pairs}. These conditions do not depend on the choice of $\ek$ and $\dk$.

\begin{defn}\label{def eps small}
A number $\ek>0$ is called \textbf{small enough} for $k$ if it verifies the following three conditions:
\begin{enumerate}
\item The neighbourhood $\Phi(X_{\ek})$ of the origin in $\C^2$ is small enough for $k$. 

\item $\forall ~0<\varepsilon<\ek$, $\Sk$ intersects transversally $\Sp_\varepsilon$. 

\item \label{condeps3} The number $\ek>0$ is as in the fibration Theorem \ref{thm fibration}, and such that the intersection $\Vf\cap \Vg$ is transverse in $X_{\ek}\setminus \{0\}$, and for all $0<\varepsilon<\ek$, the intersection $(\Vf \cap \Vg)\cap \Sp_\varepsilon$ is transverse.
\end{enumerate}
\end{defn}

From now on, $\ek>0$ is chosen to be small enough for $k$.

\begin{lemma}\label{orientation Sk}
The germ $(\Sk,0)$ is an isolated singularity of real analytic $4$-variety. Furthermore the representative $\Sk\cap X_{\ek}$ of $(\Sk,0)$ is orientable. 
\end{lemma}
\begin{proof}
Thanks to Lemma \ref{lemma discriminant}, the variety $\Sk\setminus (\Vf \cap \Vg)\cap \X_{\ek}$ is a smooth locally trivial fibration over $\Zk\setminus \{0\}\cap \Phi(\X_{\ek})$, with fiber a smooth complex curve of the form $f^{-1}(x_0)\cap g^{-1}(y_0)\cap \overline{\X_{\ek}}$. 

The smoothness of $\Sk\cap X_{\ek}$ at $\Vf \cap \Vg \setminus \{0\}$ comes from the fact that on $\Vf \cap \Vg \setminus \{0\}$, because of the conditions of Definition \ref{companion}, $\Vf$ and $\Vg$ are smooth surfaces intersecting transversally. Hence if $p \in \Vf \cap \Vg \setminus \{0\}$, there are local coordinates $(u,v,w)$ at $p$ such that $\Vf=\{u=0\}$ and $\Vg=\{v=0\}$. Locally, $$ \left\{ \begin{array}{l}
f=I_f(u,v,w) \cdot u\\

g=I_g(u,v,w)\cdot v
\end{array} \right. $$ where $I_f,I_g$ are units at $p$. Hence, at $p$, $$\Sk =\{I(u,v,w)u=|v|^k\}$$ where $I$ is a unit at $p$. This shows that $\Sk$ is smooth along $\Vf \cap \Vg \setminus \{0\}$. 

Now, an orientation of $\Sk\setminus (\Vf\cap \Vg)$ is built by orienting the fibers using the complex structure, and then taking the product of the orientations of the base and fibers. This provides an orientation of $\Sk$.
\end{proof}

From now on, we fix a representative of the link of $\Sk$ to be $$\partial \Sk := \Sk\cap \Sp_\ek.$$

Let us now build an appropriate representative of the Milnor fiber of $f$.

Condition \ref{condeps3} of Definition \ref{def eps small} on $\ek$ implies that $\exists$ $\delta_1,\delta_2>0$ such that $\Phi_{|S_\ek}$ induces a smooth locally trivial fibration above $D_{\delta_1}\times D_{\delta_2}\subset \C^2$. Note that it implies that $D_{\delta_1}\times D_{\delta_2}\subset \Phi(X_\ek)$.

Set $\dk > 0$ to be as in Theorem \ref{thm fibration} relatively to $\varepsilon$, and such that $\dk \leqslant \delta_1$ and $\dk^{1/k} \leqslant \delta_2$. The representative of the Milnor fiber that we choose will be $$F:=\overline{\{f=\dk\}\cap X_\ek}$$ with boundary $$\partial F=f^{-1}(\dk)\cap S_\ek.$$

We can now state and prove the following:

\begin{prop}
\label{prop diffeomorphic boundaries}
In the setting described in this section, the oriented manifolds $\partial \Sk$ and $\partial F$ are orientation-preserving diffeomorphic.
\end{prop}

\begin{proof}
We are going to build an isotopy from $\partial F$ to $\partial \Sk$ in the smooth locus $\sm{X_\varepsilon}$ of $X_\varepsilon$, using vector fields. This isotopy will be built in three steps, represented schematically in figure 

In the sequel, we will use the notion of gradient for some functions defined on parts of $\sm{X_\varepsilon}$. Take an embedding $X_\varepsilon \subset \C^N$. Then this gradient will be defined with respect to the restriction of the usual hermitian form of $T_p\C^N$ on $T_p\sm{X_\varepsilon}$, for any $p\in \sm{X_\varepsilon}$.

First, observe that there is an isotopy from $F$ to the manifold with corners \begin{equation}\label{eq F corners}
\Fs:=\{f=\delta\}\cap \{|g|\leqslant \delta^{1/k}\} \cap \overline{X_\varepsilon}
\end{equation}

This isotopy follows from integrating the vector field provided by the gradient of the function $$|g|\colon F\cap \{|g|\geqslant\delta ^{1/k}\} \rightarrow \R_+$$ To see that this vector field is everywhere non zero on the compact set $F\cap \{|g|\geqslant\delta ^{1/k}\}$, observe that the hypothesis $k \geq k_0$ ensures that $F\cap \{|g|\geqslant\delta ^{1/k}\} \cap C_\Phi=\emptyset$, which means that the levels of $f$ and $g$ intersect transversally everywhere on $F\cap \{|g|\geqslant\delta ^{1/k}\}$, guaranteeing that the same is true for $f$ and $|g|$.

In the same spirit, there is an isotopy from $\Sk$ to the manifold with corners \begin{equation}\label{eq Sk square}\Sks:=\Sk\cap \{|g|\leqslant\delta ^{1/k}\}\cap X_\varepsilon
\end{equation}
This one is obtained, again, by integrating the gradient of the function $$|g|\colon \Sk\cap \{|g|\geqslant\delta ^{1/k}\} \rightarrow \R_+$$ which is an everywhere non-zero vector field on the compact $\Sk\cap \{|g|\geqslant\delta ^{1/k}\}$.

Now, observe that we have the following two decompositions for the boundaries :
$$
\partial \Fs= \left(\{f=\delta\}\cap \{|g|= \delta^{1/k}\} \cap \overline{X_\varepsilon}\right) \bigcup\limits_{\{f=\delta\}\cap \{|g|= \delta^{1/k}\} \cap S_\varepsilon} \left(\{f=\delta\}\cap \{|g|\leqslant \delta^{1/k}\} \cap S_\varepsilon\right)
$$

$$
\partial \Sks= \left(\{f=\delta\}\cap \{|g|= \delta^{1/k}\} \cap \overline{X_\varepsilon}\right) \bigcup\limits_{\{f=\delta\}\cap \{|g|= \delta^{1/k}\} \cap S_\varepsilon} \left(\{f=|g|^k\}\cap \{|g|\leqslant \delta^{1/k}\} \cap S_\varepsilon\right)
$$

We now build an isotopy on $S_\varepsilon$ taking the second part of the decomposition of $\partial \Fs$ to the one of $\partial \Sks$, and preserving $\{f=\delta\}\cap \{|g|= \delta^{1/k}\} \cap S_\varepsilon$. It will, again, be obtained by integrating an everywhere non-zero vector field. This one is built in the following way:

At $p\in \{Im(f)=0\}\cap \{|g|^k\leqslant Re(f)\leqslant \delta\}\cap S_\varepsilon$, the vector is given by the gradient of the function $$Re(f)\colon \{g=g(p)\} \cap S_\varepsilon \rightarrow \R$$

The fact that this vector field is nowhere zero on the compact $\{Im(f)=0\} \cap \{|g|^k\leqslant Re(f)\leqslant \delta\}\cap S_\varepsilon$ comes from the condition asked on $\delta$. Indeed $\delta$ is chosen so that $\Phi$ induces a smooth locally trivial fibration over $D_\delta\times D_{\delta^{1/k}}$, ensuring that at a point of this compact, the levels of $f$ and $g$ intersect transversally, guaranteeing the same for the levels of $Re(f)$ and $g$.

The combination of these three isotopies provides an isotopy between $\partial \Sk$ and $\partial F$, hence an orientation-preserving homeomorphism. To conclude this proof, it is enough to invoke the fact that two $3$-manifolds that are orientation-preserving homeomorphic are also orientation-preserving diffeomorphic. See e.g. Munkres \cite{Mun60}.

Let us note however that if one wants to build a diffeomorphism, it can be obtained using the three isotopies that we composed in this proof, and smoothing the corners at every step. 

\end{proof}

\section{The tower of morphisms}\label{section tower}
Proposition \ref{prop diffeomorphic boundaries} implies that the computation of $\partial F$ can be replaced by the one of $\partial \Sk$. We will carry out this computation through a resolution of $(\Sk,0)$ with particular properties. Namely, in the next sections \ref{section first modif} to \ref{resolution step}, we build step by step a proper morphism $$\Pi\colon (\Sb,\Ex)\rightarrow (\Sk,0)$$ where $\Ex$ is a simple configuration of compact smooth real analytic oriented surfaces (see Definition \ref{def simple conf of surfaces, plumbing graphs.}) in a oriented real analytic manifold $\Sb$ of dimension $4$. 

This morphism will be an analytic isomorphism outside $E$, and will therefore allow us to identify the boundary $\partial \Sk$ with its preimage by $\Pi$. We then prove that this preimage can be seen as the boundary of a tubular neighbourhood of $E$ in $\Sb$, which is a graph manifold determined by the configuration $E$, as explained in Theorem \ref{thm corresp boundary plumbing graph}.

At every step, the preimage of the origin will be a configuration of real analytic oriented surfaces, that we will represent by its dual graph, with some decorations. Note that until the penultimate step, some of these surfaces may be singular. However at every step it will be possible to associate a dual graph to the preimage of the origin.

The first level of this morphism will come from a modification $\rx$ of $X$, respecting certain conditions relatively to $f$ and $g$. We will consider the strict transform $\Skt$ of $\Sk$ by $\rx$. Denote the restriction of $\rx$ to $\Skt$ by $$\tr\colon \left(\Skt,\Courbes\right) \rightarrow \left(\Sk,0\right).$$

The second step is the normalization of $\Skt$ $$\norm \colon \left(\Sn,\Courbesn\right) \rightarrow \left(\Skt,\Courbes\right).$$

Then, using local equations for $\Skt$, we build local morphisms $\kappa_p$ from complex surfaces to a finite number of open sets covering $(\Skt,\Courbes)$, and use their liftings to the normalizations of the source and the target to build a global morphism $$\K \colon \left(\St,\Courbesk\right) \rightarrow \left(\Sn,\Courbesn\right)$$ from a $4$-dimensional variety $\St$ to $\Sn$, that will be a diffeomorphism outside $\Courbesn$.

The variety $\St$ will have controlled isolated singularities, namely of Hirzebruch-Jung type. A standard resolution of such singularities is explained in Appendix \ref{appendix resol HJ}, and requires only few data to be computed, leading to the morphism of the final step$$\pi \colon \left(\Sb,E\right)\rightarrow \left(\St,\Courbesk\right).$$

The construction of the composed morphism $\Pi$ is summarized in the following diagram
\begin{center}

\begin{tikzpicture}

\filldraw 
(3,-1.3) circle (0pt) node[] {$(\Zk,0)$} 
(3,0) circle (0pt) node[] {$(\Sk,0)$} 
(3,1.3) circle (0pt) node[] {$\left(\Skt,\Courbes\right)$} 
(3,2.9) circle (0pt) node [] {$\left(\Sn,\Courbesn\right)$}
(3.2,4.4) circle (0pt) node[] {$\left(\St,\Courbesk\right)$} 
(3,5.5) circle (0pt) node[] {$\left(\Sb,E\right)$}
(3.9,0) circle (0pt) node[] {$\subset$}
(3.9,1.3) circle (0pt) node[] {$\subset$}
(3.9,-1.3) circle (0pt) node[] {$\subset$}
(4.8,0) circle (0pt) node[] {$(\X,0)$}
(4.8,1.3) circle (0pt) node[] {$(\Xt,\Do)$}
(4.8,-1.3) circle (0pt) node[] {$(\C^2,0)$}
(6.4,0) circle (0pt) node[] {$(\R^+,0)$}
;
\path[-stealth]
(3,5.2) edge node[right]{$\resol$} (3,4.6)
(3,4.1) edge node[right]{$\K$} (3,3.2)
(3,2.4) edge node[right]{$\norm$} (3,1.7)
(3,1) edge node[right]{$\tr$} (3,0.2)
(3,-0.2) edge node[right]{$\Phi$} (3,-1)
(4.6,1) edge node[right]{$\rx$} (4.6,0.2)
(4.6,-0.2) edge node[right]{$\Phi$} (4.6,-1)
(5.3,0) edge node[below]{$\dist$} (5.8,0)
(2.3,5.5) edge [bend right=60] node [left] {$\Pi$} (2.4,0)
;

\end{tikzpicture}

\end{center}

\subsection*{Algorithmic aspect} 
We will decorate the different dual graphs
$$\Gc,\Gk,\Gresol$$ 
obtaining the decorated graphs 
$$\Gcd,\Gkd,\Gresold. $$
The decorations of these graphs will be of different natures. 

They are built and decorated in order to get the plumbing graph
 $$\Gresold=\Gdp{E}{\Sb}$$
for the boundary of a tubular neighbourhood of $E$ in $\Sb$, which is the aim of this study.

It is very important to notice that, at every step, some of the information needed to compute a decorated graph is encoded in the previous one. Namely:
\hfill
\begin{itemize}
\item $\Gresold$ is entirely determined by $\Gkdt$.
\item $\Gkdt$ is entirely determined by $\Gcdt$ and some information relative to the singularities of $X$, namely the collection of switches introduced in Subsection \ref{subsection non rational curves}, and the additional covering data mentioned in Subsection \ref{subsection graphe gkdt}. In general, accessing that information may represent a difficult task, for which one should have, among other things, a good understanding of the structure of the initial modification $\rx$. However this does not stop us from carrying out the proof of the general result. In some cases, that additional information is unnecessary, as for example in the case where $X$ is smooth (see \cite{NemSzi12}), or in the toric case (see \cite{cur19.2}).
\end{itemize}

\textbf{In conclusion}, this relation between the graphs, which in some cases allows to deduce one graph from the previous one, permits us to build a method for the computation of $\partial F$, whose starting point will be the decorated graph $\Gcdt$ together with the required additional information relative to the singularities of $X$. The construction and definition of $\Gcdt$ and $\Skt$ is the object of the next section.
\section{First modification of $\Sk$}\label{section first modif}
Recall that a \textbf{modification} of a complex analytic variety is a proper bimeromorphic morphism whose target is the given variety. We introduce in this section a complex analytic modification $\rx:\Xt\rightarrow \X$ of $X$, that will verify some conditions relative to $f$ and $g$:
\begin{defn}{\bf (Modification adapted to $\Phi$.)}\label{modif adaptee}\index{Modification!adapted to a map $\Phi$}

\noindent A modification $\rx\colon (\Xt,\rx^{-1}(0)) \rightarrow (X,0)$ is said to be \textbf{adapted} to the morphism $\Phi=(f,g)$, where $g$ is a companion of $f$, if the two following conditions are simultaneously satisfied:
\begin{enumerate}
\item \label{modif locus} The modification $\rx$ is an isomorphism outside of $(\Vf\setminus \Vg) \cup \{0\}$. Or equivalently, $\rx$ may not be an isomorphism only at $\Vf\setminus \Vg$ or $\{0\}$.

\item \label{sncd} Denote $\Dt:=\rx^{-1}(\V)$.
Denote by $\Do$ the union of the irreducible components of $\Dt$ sent on the origin by $\rx$. Denote by $\Dfex$ the union of the components of $\mathbb{D}$ sent on curves in $V(f)$, and by $\Vft$ and $\Vgt$ the strict transforms of $\Vf$ and $\Vg$ by $r_X$ respectively. Finally, denote $\Df:=\Dfex \cup \Vft$.

We define two complex curves in $\Xt$ by
$$\Courbest:=\left(\Df \cap \Do\right) \cup (\Df \cap \Vgt).$$ and $$\Courbes:=\left(\Df \cap \Do\right)\cup \left(\Dfex \cap \Vgt\right),$$ the union of the compact irreducible components of $\Courbest.$ 

In order to emphasize that the global configuration of the irreducible components of the curves $\Courbes$ and $\Courbest$ is essential for us, we will say that the two curves are {\bf curve configurations}.

With these notations, the second condition imposed on $\rx$ is that the total transform $$\Dt:=\rx^{-1}(\V)$$ of $\V$ shall be a simple normal crossings divisor at every point of $\Courbest$.
\end{enumerate}
\end{defn}

The fact that $\rx$ does not modify $\Vg$ except at the origin leads to:
\begin{lemma}\label{lemma preim origin rx}
The preimage of the origin by $\rx$ is $$\rx^{-1}(0)=\Do\cup (\Dfex\cap \Vgt).$$
\end{lemma}


%
%


\begin{rk}
The curves $\Courbes$ and $\Courbest$ will play a key role in the sequel. They are in general complicated to compute, but in the toric setting, we can express them with only combinatorial considerations (see \cite{cur19.2}).
\end{rk}

\begin{defn}\label{orientations courbes initiales}
Orient the irreducible components of $\Courbes$ and $\Courbest$ by their complex structures in $\Xt$.
\end{defn}

Condition \ref{modif locus} of Definition \ref{modif adaptee} implies that what we build is not an embeddded resolution of the surface $\V$. Indeed we do not modify it along $\Vf\cap \Vg \setminus \{0\}$, where it has simple singularities, i.e. it looks locally like a transverse intersection of smooth surfaces. 

The existence of a modification of $X$ adapted to $\Phi$ is guaranteed by the following

\begin{lemma}{\bf (Resolution Lemma), \cite[p.633]{Sza94}}\label{Resolution lemma}

Let $X$ be an irreducible variety over an algebraically closed field of characteristic $0$, $V$ a codimension one subvariety of $X$, and $U$ an open subset of the smooth locus of $X$ such that $U\cap V$ has smooth components crossing normally. Then there is a projective morphism $\rx \colon \Xt \rightarrow X$ which satisfies the following conditions:
\begin{enumerate}
\item $\rx$ is a composition of blowing ups of smooth subvarieties.
\item $\rx$ induces an isomorphism over $U$.
\item $\Xt$ is smooth.
\item $\rx^{-1}(V\cup(X\setminus U))$ is a divisor. Moreover, it has smooth components crossing normally.
\end{enumerate}
\end{lemma}

In our context, one can for example take the open subset $U$ to be $X\setminus  \si{\Vf}$.
\begin{rk}\label{remark modif}
\begin{enumerate}
\item The union of the components of $\Df$ is what we call the \textbf{mixed transform} of $\Vf$ by $\rx$. Indeed it is not exactly the strict transform of $\Vf$ by $\rx$, as it also contains the components of $\Dfex$ that are in the exceptional divisor, in particular the preimage of the singular locus $\si{\Vf}$ of $\Vf$.
\item Note that $\Vgt=\overline{\rx^{-1}(\Vg\setminus 0)}$ because the origin is the only point of $\Vg$ at which we allow $\rx$ not to be an isomorphism.
\item The Resolution Lemma \ref{Resolution lemma} implies that in general one can always build a modification $\rx$ such that $\Dt$ is globally a Simple Normal Crossings Divisor, which will imply Condition \ref{sncd} on $\rx$. However, what we really need in order to proceed is this weaker condition. 
\item In the same spirit, unlike what is required in \cite[6.1.2]{NemSzi12}, we allow the morphism $\rx$ to modify some curves in $(\Vf \setminus \Vg)\cup \{0\}$ along which $\Vf$ may be smooth. This point and the previous one will be important for our analysis of Newton nondegenerate surface singularities in the toric setting.
\item We get a partition of the set of irreducible components of the total transform $\Dt$ in three parts: the set of irreducible components of $\Df,\Vg$ or $\Do$. By abuse of notation, we will sometimes refer to an irreducible component of $\Dt$ as ``in'' ($\in$) one of these three surfaces, instead of ``subset of'' ($\subset$).
\end{enumerate}
\end{rk}
\begin{nota}\label{multiplicities}
Denote $$\ft:=f \circ \rx,~ \gt:=g \circ \rx : \Xt \rightarrow \C$$ the pullbacks of $f$ and $g$ to $\Xt$. In the sequel, if $D_i$ is an irreducible component of $\Dt$, we will denote $$m_i:=mult_{D_i}(\ft),~ n_i:=mult_{D_i}(\gt)$$ the respective multiplicities of $\ft,\gt$ along this component.

Note that if $D_1\in \Vg$, then necessarily $m_1=0$ and $n_1=1$.
\end{nota}

\noindent The firt step of the construction of the desired morphism $\Pi$ is the modification of $\Sk$ obtained by restricting $\rx$ to the strict transform $$(\Skt,\Courbes):=\left(\overline{\rx^{-1}(\Sk\setminus \{0\})},\Courbes\right)\subset (\Xt,0).$$ of $\Sk$ by $\rx$. It will have non-isolated singularities, being from this point of view more complicated than $(\Sk,0)$. However we can domesticate these singularities. This is the point of the next section. 

\section{Normalization and local models}\label{section norm loc models}

In this section we give ``local complex models'' for the normalization $\Sn$ of $\Skt$, which will allow us to get back to the context of singularities of complex surfaces.

Denote $\tr:= {\rx}_{|\Skt} :\Skt\rightarrow \Sk$, and recall the notations $m_i,n_i$ introduced in Notation \ref{multiplicities}. 

\begin{prop} \label{preimage de 0 dans skt} Let $k\in 2\N$ be such that for every component $D_1\in \Do$, $kn_1>m_1$. Then:
\hfill
\begin{enumerate}

\item The preimage in $\Skt$ of the origin in $\Sk$ is \begin{equation}\label{Courbes im rec de 0}\tr^{-1}(0)=\Courbes=(\Df\cap \Do)\cup (\Dfex\cap \Vgt).
\end{equation}
\item The support of the total transform of $\Vg \cap \Sk$ by $\tr$ is equal to $\Courbest=(\Df\cap \Do)\cup(\Df\cap \Vg)$. In other words, denoting $$\gts:=g\circ \tr\colon \Skt\rightarrow \C$$ the pullback of $g$ to $\Skt$, we have \begin{equation}\label{Courbest transf tot g}
\gts^{-1}(0)=\Courbest.
\end{equation}
\end{enumerate}
\end{prop}

The statement of Proposition \ref{preimage de 0 dans skt} implies that we need new conditions on $k$ in order to proceed to the computation of the boundary of $\Sk$. In the sequel, more will appear. Let us make those conditions explicit.
\begin{conditionk} \emph{\textbf{on $k$.}} \label{k large enough}
The integer $k$ must be even and large enough in the sense of Lemma \ref{lemma discriminant}.
\end{conditionk}

\begin{conditionk} \emph{\textbf{on $k$.}} \label{cond k bons points.}
The integer $k$ must be such that, for every component $D_1\in \Do$, $kn_1>m_1$.
\end{conditionk}

\begin{nota}\label{nota strict part}
Let $H\subset \C^3_{u,v,w}$. Denote $\str[H]$ the \textbf{strict part} of $H$ defined as the closure $$\str[H]:=\overline{H\setminus \{uvw=0\}}.$$
\end{nota}

\begin{proof}[Proof of Proposition \ref{preimage de 0 dans skt}]

Lemma \ref{lemma preim origin rx} implies that for the first statement we have to prove \begin{equation}  \label{eq Do Df}\Skt\cap \Do = \Df \cap \Do,\end{equation} \begin{equation} \label{eq Dfex Vgt} \Dfex\cap \Vgt \subset \Skt \end{equation} and, for the second statement, \begin{equation}\label{eq transf str Skt}
\Vgt \cap \Skt = \Vgt \cap \Vft.
\end{equation}

Let us prove \eqref{eq Do Df}. Let $p$ be a generic point of a component $D_1\in \Do$. Then one can take local holomorphic coordinates $(u,v,w)$ on a neigbourhood $U_p$ of $p$ in $\Xt$ such that $D_1\cap U_p=\{u=0\}\cap U_p$. Then on $U_p$, $$ \left\{ \begin{array}{l}
\ft=I_f(u,v,w) u^{m_1}\\
\gt=I_g(u,v,w) u^{n_1}
\end{array} \right. $$ where $I_f,I_g$ are units at $p$.

Then $$\Skt\cap U_p=\overline{\{I_f\cdot u^{m_1}=|I_g|^k|u|^{kn_1}\}\setminus V(u)} \cap U_p.$$
Now the condition $kn_1>m_1$ implies that $\Skt$ does not contain $p$.


Consider a point $p\in D_1 \cap D_2$, where $D_1 \in \Do, D_2 \in \Vgt$, and $p$ is on no other component of $\Dt$. Then, as $\Dt$ is a simple normal crossings divisor at $p$, we can set local holomorphic coordinates $(u,v,w)$ on a neighbourhood $U_p$ of $p$ in $\Xt$ such that $D_1=\{u=0\},D_2=\{v=0\}$. Then on $U_p$, $$ \left\{ \begin{array}{l}
\ft=I_f(u,v,w) u^{m_1}\\

\gt=I_g(u,v,w) u^{n_1}v^{n_2}
\end{array} \right. $$ where $I_f,I_g$ are units at $p$.

On $U_p$, $\Skt$ is then defined by $\overline{\{\ft=|\gt|^k\}\setminus V(u)}$, that is, $$\Skt\cap U_p=\str[\{I_f\cdot u^{m_1}=|I_g|^k|u|^{kn_1}|v|^{kn_2}\}]\cap U_p.$$

Then the condition $kn_1>m_1$ shows that $\Skt$ does not contain $p$.

Finally, consider a point $p\in D_1 \cap D_2$, where $D_1 \in \Do, D_2 \in \Df$, and $p$ is on no other component of $\Dt$. Then, as $\Dt$ is a simple normal crossings divisor at $p$, we can set local holomorphic coordinates $(u,v,w)$ on a neighbourhood $U_p$ of $p$ in $\Xt$ such that $D_1=\{u=0\},D_2=\{v=0\}$. Then on $U_p$, $$ \left\{ \begin{array}{l}
\ft=I_f(u,v,w) u^{m_1}v^{m_2}\\

\gt=I_g(u,v,w) u^{n_1}
\end{array} \right. $$ where $I_f,I_g$ are units at $p$.

On $U_p$, $\Skt$ is then defined by $\overline{\{\ft=|\gt|^k\}\setminus V(u)}$, that is, $$\Skt\cap U_p=\str[\{I_f\cdot u^{m_1}v^{m_2}=|I_g|^k|u|^{kn_1}\}]\cap U_p$$ and $\Skt$ contains $p$.

Equations \eqref{eq Dfex Vgt} and \eqref{eq transf str Skt} are proved in the same way, taking local equations for $\Skt$.
\end{proof}
We want to keep track of the total transform of $\Vg \cap \Sk$ in $\Skt$ because the pullback of $g$ by $\Pi$ will be the adapted function that we use to compute the self-intersections of the irreducible components of $\Ex$ in $\Sb$, as explained in Lemma \ref{lemma compute self-intersections}.

We will now determine local equations of $\Skt$ around points of $\Courbest$, in order to define and study the morphisms $\norm$ and $\K$ mentioned in Section \ref{section tower}. 

Let $p\in \Courbest$. Because $\Dt$ is a normal crossings divisor at each point of $\Courbest$, $p$ will be on no more than three irreducible components of $\Dt$ at the same time. We face two different situations. In each of them we first build local coordinates giving $\Skt$ a simple equation, and use it to provide a ``local algebraic model'' of this surface. We then study the lifting of this model to the normalizations of the source and the target. This lifting, globally, provides the morphism $\K$.
\begin{discussion}
In the sequel, we show among other things that there is a finite open covering of $(\Skt,\Courbest)$ such that, on each one of these open sets, $\Skt$ can be presented as an affine real-algebraic variety. This allows us, using the definition of normalization of a real-algebrai variety (see \cite[p. 75]{BocCosRoy}), to speak about the normalization $\Sn$ of $\Skt$, which is a real analytic variety. Denote by $\norm \colon \Sn \rightarrow \Skt$ the \textbf{normalization morphism} of $\Skt$.
\end{discussion}

Taking normalizations, we are going to make great use of the following classical proposition (see again \cite{BocCosRoy}):
\begin{prop}\label{prop lifting normalization}
Let $V=Specm(A)$ and $W=Specm(B)$ be two affine $\mathbb{K}$-algebraic varieties, $\mathbb{K}=\R$ or $\C$, and $\varphi \colon V \dashrightarrow W$ be a birational map. Then there is a unique birational map $\varphi ^N\colon V^N \dashrightarrow W^N$ at the level of normalizations such that the following diagram commutes:
\begin{center}
\begin{tikzpicture}
\filldraw
(0,0) circle (0pt) node[] {$V^N$}
(2,0) circle (0pt) node[] {$W^N$}
(0,-1.5) circle (0pt) node[] {$V$}
(2,-1.5) circle (0pt) node[] {$W$}
(1,-0.75) node {$\circlearrowright$}
;
\path[-stealth] (0,-0.3) edge node [left]{$N_V$} (0,-1.2)
(2,-0.3) edge node[right]{$N_W$} (2,-1.2);
\path[dashed][-stealth] (0.3,-0.1) edge node[above]{$\varphi^N$} (1.6,-0.1)
(0.3,-1.5) edge node[below]{$\varphi$} (1.7,-1.5)
;
\end{tikzpicture}
\end{center}

Furthermore, let $\psi$ be an inverse of $\varphi$. If $X_1,\cdots,X_n$ and $Y_1,\cdots,Y_p$ are respectively generators of $A$ and $B$ such that for any $i$, $\varphi^*(Y_i)$ is integral over $A$ and for any j, $\psi^*(X_j)$ is integral over $B$, then $\varphi^N$ is an isomorphism of algebraic varieties.
\end{prop}

From now on, if $x\in \R_+, n\in \N^*$, $x^{1/n}$ will denote the unique $n$-th root of $x$ in $\R_+$.

\subsection{Generic points of $\Courbest$}\label{gen pt}

Take a point $p\in D_1 \cap D_2$, where $D_1 \in \Do$ or $\Vgt$, $D_2 \in \Df$, with the notations of Definition \ref{modif adaptee}, and $p$ is on no other component of $\Dt$. Then, as $\Dt$ is a simple normal crossings divisor at $p$, we can set local holomorphic coordinates $(u',v',w')$ on a neighbourhood $U'_p$ of $p$ in $\Xt$ such that $D_1=\{u'=0\},D_2=\{v'=0\}$. Then locally: $$ \left\{ \begin{array}{l}
\ft=I_f(u',v',w') u'^{m_1}v'^{m_2}\\

\gt=I_g(u',v',w') u'^{n_1}
\end{array} \right. $$ where $I_f,I_g$ are units at $p$.

Now, choose $\phi,\psi$ units at $p$ such that $\phi^{n_1}=I_g$ and $\psi^{m_2}=I_f\cdot \phi^{-m_1}$, and build new local coordinates $(u,v,w)$ around $p$ defined by $$\left\{ \begin{array}{l}
u=\phi \cdot u'\\
v=\psi\cdot v'\\
w=w'

\end{array} \right. $$ on a neighbourhood $U_p\subset U'_p$ where $\phi$ and $\psi$ are non-zero.

On $U_p$, $\Skt$ is then defined by $\overline{\{\ft=|\gt|^k\}\setminus V(u)}$, that is, \begin{equation}\label{eq Skt gen}
\Skt\cap U_p=\str[\{u^{m_1}v^{m_2}=|u|^{kn_1}\}]\cap U_p.
\end{equation}

Note that Condition \ref{cond k bons points.} on $k$ ensures that $p$ is in $\Skt$.

Now, we can provide a birational morphism:
$$\kappa_p\colon \{(x,y,z)\in \C^3, x^{m_1}=y^{m_2}\}\rightarrow \str[\{(u,v,w)\in \C^3,u^{m_1}v^{m_2}=|u|^{kn_1}\}]$$

given by
\begin{equation}\label{morph gen}
\left\{\begin{array}{l}
u=x\\
v=y^{-1}|x|^{kn_1/m_2}=y^{-1}|y|^{kn_1/m_1} \\
w=z
\end{array} \right.~~~~~~~~~
\left\{\begin{array}{l}
x=u\\
y=v^{-1}|u|^{kn_1/m_2}\\
z=w
\end{array} \right.
\end{equation}

If $kn_1/m_2\in 2\N^*$, the map $\kappa_p$ is birational. 

We then require
\begin{conditionk} \emph{\textbf{on $k$.}} \label{fractions kn1 m2 sont entiers pairs}
For every $\kappa_p$ to be birational, the integer $k$ must be such that, $\forall ~ D_1\in \Do$ that intersects a component $D_2\in \Df$, $$\frac{kn_1}{m_2}\in 2\N.$$
\end{conditionk}
In addition to the birationality of $\kappa_p$, the identity $$y^{m_2}=u^{m_1}$$ shows that the coordinates $x,y,z, \overline{x},\overline{y}$ and $\overline{z}$ of the source are integral on the ring of regular functions of $\Skt\cap U_p$. Reciprocally, the identity $$v^{m_1}=y^{-m_1}|y|^{kn_1}$$ shows that, thanks to Condition \ref{cond k bons points.} on $k$, the coordinates of the target are integral on the ring of functions of the source. 

Note that these two identities taken together, with the Condition \ref{cond k bons points.} on $k$, ensure that $\kappa_p$ is a homeomorphism.

Setting $\Upk:=\kappa_p^{-1}(U_p)$ a neighbourhood of the origin in $\C^3$ and using Proposition \ref{prop lifting normalization}, we get:

\begin{lemma}\label{diffeo gen}
The homeomorphism $\kappa_p$ induces an isomorphism of real analytic varieties $$\overline{\kappa_p}\colon \left(\{(x,y,z)\in \C^3, x^{m_1}=y^{m_2}\}\cap \Upk\right)^N \xrightarrow{\sim} \left(\Skt \cap U_p\right)^N$$ at the level of normalizations.
\end{lemma}

\begin{discussion} {\bf Orientation of $\Skt$ and the model near a generic point.}\label{orient gen}

In section \ref{Sk}, we provided a description of the orientation of $\Sk$. This orientation is pulled-back by $\tr$ to provide an orientation of $\Skt$, making $\tr$ orientation-preserving.

Another way to retrieve this orientation on an open $U_p$ as in subsection \ref{gen pt} is to observe that it is the one that makes $\partial \Skt:=\tr^{-1}(\partial \Sk)\simeq \partial \Sk$ orientation-preserving diffeomorphic to $\partial F$, where the last equivalence symbol denotes an orientation-preserving diffeomorphism.

The modification $\rx$ induces a biholomorphic morphism over $F$, and is hence an orientation-preserving diffeomorphism from $\Ft:=\rx^{-1}(F)$ to $F$, both being oriented by their complex structures. $\Ft$ is a complex manifold, and $$\Ft\cap U_p=\{I(u,v,w)u^{m_1}v^{m_2}=\delta\}$$ for some holomorphic unit $I$ at $p$. The complex orientation of $\Ft\cap U_p$ is given by $du\wedge d\overline{u}\wedge dw\wedge d\overline{w}$ (or, equivalently, $dv\wedge d\overline{v}\wedge dw\wedge d\overline{w}$, but for our purposes the first expression is more natural).

Now, the orientation of $\Skt$ on $U_p$ is the one that is compatible with this orientation of $\Ft$, that is, the one that is given on the smooth part of $\Skt\cap U_p$ by $du\wedge d\overline{u}\wedge dw\wedge d\overline{w}$.

One can see now that the morphism $\kappa_p$ is orientation-preserving, \textbf{if the source is taken with its complex orientation} $dx\wedge d\overline{x}\wedge dz\wedge d\overline{z}$.

Furthermore, the normalizations of the source and target of $\kappa_p$ are now canonically oriented by the pullbacks of these orientations, making, again, $\overline{\kappa_p}$ orientation-preserving as soon as the source is oriented by its complex structure.

Finally, remind that in Definition \ref{orientations courbes initiales}, we oriented $\Courbes\cap U_p=\{u=v=0\}\cap U_p$ via its complex structure, i.e. by $dw\wedge d\overline{w}$. Pulling back this orientation to its preimage by $\kappa_p$ gives, again, the complex orientation of the axis $\{x=y=0\}\subset \Upk$. 

Therefore, the pullback of the orientation of $\Courbest \cap U_p$ by $\norm\circ \kappa_p$ \textbf{orients its preimage \linebreak $(\norm\circ \kappa_p)^{-1}(\Courbest\cap U_p)$ by its complex structure.}
\end{discussion}

\subsection{Double points of $\Courbest$}
\label{triple points}
Let $p \in D_1 \cap D_2 \cap D_3$, where $D_1 \in \Do, D_2 \in \Df, D_3 \in \Dt$, with the notations of Definition \ref{modif adaptee}. Set local holomorphic coordinates $(u',v',w')$ on a neighbourhood $U'_p$ of $p$ in $\Xt$ such that $D_1=\{u'=0\},D_2=\{v'=0\}, D_3=\{w'=0\}$. Then we can write locally: $$ \left\{ \begin{array}{l}
\ft=I_f(u',v',w') u'^{m_1}v'^{m_2}w'^{m_3}\\
\gt=I_g(u',v',w') u'^{n_1}w'^{n_3}
\end{array} \right. $$ where $I_f,I_g$ are units at $p$.

\noindent In the same way as in \ref{gen pt}, we obtain local coordinates $(u,v,w)$ on a neighbourhood $U_p\subset U'_p$ such that \begin{equation}\label{eq Skt double}
\Skt\cap U_p=\str[\{u^{m_1}v^{m_2}w^{m_3}=|u|^{kn_1}|w|^{kn_3}\}]\cap U_p.
\end{equation}

\noindent and $\Skt$ contains $p$ because, again, of Condition \ref{cond k bons points.} on $k$ for $D_1$.

We will provide, again, a local algebraic model for $\Skt$ around $p$, that will depend on the nature of the double point $p$.
 
\begin{defn} \label{def convention signes}(See \cite[section 6.2]{NemSzi12})\index{Double points of $\Courbest$ of type $\oplus$/$\ominus$}

The double point $p$ is said to be \textbf{of type $\plus$} if $D_3\in \Dg$ or $\Do$, and \textbf{of type $\minus$} if $D_3\in \Df$.

That is, $p$ is of type $\plus$ if and only if $n_3\neq 0$.
\end{defn}

%
%
\begin{rk}
We separate the two types of points using signs, anticipating the fact that they will give different signs of edges in the final plumbing graph for $\dF$. Némethi and Szil\'ard use respectively $1$ and $2$, instead of $\plus$ and $\minus$, to designate the two types of double points.
\end{rk}

\subsubsection{Points of type $\plus$.}
If $p$ is of type $\plus$, things go more or less the same way as in the case of generic points. Namely, there is a birational map 
$$\kp_p\colon \{(x,y,z)\in \C^3, y^{m_2}=x^{m_1}z^{m_3}\}\rightarrow \str[\{u^{m_1}v^{m_2}w^{m_3}=|u|^{kn_1}|w|^{kn_3}\}]$$ \nomenclature{$\kp_p$}{}

\noindent given by

\begin{equation}\label{eq morph plus}
\left\{\begin{array}{l}
u=x\\
v=y^{-1}|x|^{kn_1/m_2}|z|^{kn_3/m_2} \\
w=z
\end{array} \right.~~~~~~~~~
\left\{\begin{array}{l}
x=u\\
y=v^{-1}|u|^{kn_1/m_2}|w|^{kn_3/m_2}\\
z=w
\end{array} \right.
\end{equation}
Condition \ref{fractions kn1 m2 sont entiers pairs} on $k$ shows that this map is birational. Furthermore it is a homeomorphism.

Moreover, as in \ref{gen pt}, using $y^{m_2}=u^{m_1}w^{m_3}$ and reciprocally, $v^{m_2}=|x|^{kn_1}x^{-m_1}|z|^{kn_3}z^{-m_3}$, still with $kn_i>m_i$, and denoting $\Upk:={\kp_p}^{-1}(U_p)$ a neighbourhood of the origin in $\C^3$, we have: \nomenclature{$\Upk$}{}

\begin{lemma}\label{diffeo plus}\nomenclature{$\overline{\kp_p}$}{}
The morphism $\kp_p$ is a homeomorphism, and it induces an isomorphism of real analytic varieties $$\overline{\kp_p}\colon \left(\{(x,y,z)\in \C^3, y^{m_2}=x^{m_1}z^{m_3}\}\cap \Upk\right)^N \xrightarrow{\sim} \left(\Skt \cap U_p\right)^N$$at the level of normalizations.
\end{lemma}

\begin{discussion}{\bf Orientation of $\Skt$ and its model near a point $\plus$.}\label{orient plus} \index{Orientations!of $\Skt$ and its normalization near!a point of type $\oplus$} \index{Orientations!of the model near!a point of type $\oplus$}

An argument analogous to the one developped in discussion \ref{orient gen}, with, this time, $$\Ft\cap U_p=\{I(u,v,w)u^{m_1}v^{m_2}w^{m_3}=\delta\}$$ provides again an orientation of $\Skt$ given by $du\wedge d\overline{u}\wedge dw\wedge d\overline{w}$, and $\overline{\kp_p}$ is then orientation-preserving as soon as \textbf{its source is oriented by its complex structure.}

Furthermore, the two irreducible components of $\Courbest$ that are visible here, given by $\{u=v=0\}$ and $\{w=v=0\}$, are oriented as complex curves, respectively by $dw\wedge d\overline{w}$ and $du\wedge d\overline{u}$. The morphism $\kp_p$ then orients their preimages, respectively $\{x=y=0\}$ and $\{z=y=0\}$, by their complex orientations $dz\wedge d\overline{z}$ and $dx\wedge d\overline{x}$. Hence, again, the pullback of the orientation of $\Courbest \cap U_p$ by $\norm\circ \kp$ \textbf{orients its preimage $(\norm\circ \kp)^{-1}(\Courbest\cap U_p)$ by its complex structure.}
\end{discussion}

\subsubsection{Points of type $\minus$.}
If $p$ is a double point of type $\minus$, one can say less, but still provide a local algebraic model for $\Skt$. Let us introduce the morphism:
$$\km_p\colon \{(x,y,z)\in \C^3, x^{m_1}=y^{m_2}z^{m_3}\}\rightarrow \str[\{u^{m_1}v^{m_2}w^{m_3}=|u|^{kn_1}\}]$$

\noindent given by
\begin{equation}\label{eq morph minus}
\left\{\begin{array}{l}
u=x\\
v=y^{-1}|y|^{kn_1/m_1} \\
w=z^{-1}|z|^{kn_1/m_1}
\end{array} \right.
\end{equation}

Condition \ref{cond k bons points.} on $k$ implies that $\km_p$ is regular. Furthermore it is a homeomorphism. Indeed, identifying the arguments and moduli, one gets
$$\left\{\begin{array}{l}
x=u\\
y=v^{-1}|v|^{\faktor{kn_1}{(kn_1-m_1)}} \\
z=w^{-1}|w|^{\faktor{kn_1}{(kn_1-m_1)}}
\end{array} \right.$$

\begin{rk}
However, this morphism is \textbf{not} birational as soon as $\frac{kn_1}{kn_1-m_1}$ is not an even integer. 

For example if $kn_1>2m_1$, $\frac{kn_1}{kn_1-m1}=1+\frac{m_1}{kn_1-m_1}$ is not an integer.
\end{rk}

The morphism $\km_p$, sending no irreducible component of the source to the non-normal locus of the target (see Proposition \ref{prop lifting normalization}), lifts to a morphism $$\overline{\km_p}\colon \left(\{(x,y,z)\in \C^3, x^{m_1}=y^{m_2}z^{m_3}\}\cap U \right)^N \rightarrow \left(\Skt \cap U_p\right)^N$$ at the level of normalizations.

Although $\overline{\km_p}$ is not an isomorphism, we have the following:
\begin{lemma}\label{diffeo moins}
Setting $\Upk:={\km_p}^{-1}(U_p)$, the restriction of $\overline{\km_p}$ induces an isomorphism of real analytic varieties $$\overline{{\km_p}^*}\colon \left(\{(x,y,z)\in \C^3, x^{m_1}=y^{m_2}z^{m_3}\}\cap \Upk \setminus \{0\}\right)^N \rightarrow \left(\Skt \cap U_p\setminus \{0\}\right)^N$$ at the level of ``punctured'' normalizations.
\end{lemma}

\begin{proof}
Cover $\left(\Skt \cap U_p\setminus \{0\}\right)^N=\norm^{-1}\left(\Skt \cap U_p\setminus \{0\}\right)$ with two sets: $$\left(\Skt \cap U_p\setminus \{0\}\right)^N=\norm^{-1}\left(\Skt \cap U_p\setminus V(v)\right)\bigcup\limits_{\norm^{-1}\left(\Skt \cap U_p\setminus V(v\cdot w)\right)}\norm^{-1}\left(\Skt \cap U_p\setminus V(w)\right)$$

We are going to prove that, in restriction to each of these two sets, $\overline{{\km_p}^*}$ induces a diffeomorphism.

In order to do this, consider the birational map

$$\alpha_{p,v}\colon \{(x',y',z')\in \C^3, x'^{m_1}=y'^{m_2}z'^{m_3}\}\setminus V(y')\rightarrow \str[\{u^{m_1}v^{m_2}w^{m_3}=|u|^{kn_1}\}]\setminus V(v)$$

\noindent given by

$$\left\{\begin{array}{l}
u=x'\\
v=y'^{-1}\\
w=z'^{-1}|x'|^{kn_1/m_3}
\end{array} \right.~~~~~~~~~
\left\{\begin{array}{l}
x'=u\\
y'=v^{-1}\\
z'=w^{-1}|u|^{kn_1/m_3}
\end{array} \right.$$

\noindent Thanks to Condition \ref{fractions kn1 m2 sont entiers pairs} on $k$, this map is birational. Furthermore, with Condition \ref{cond k bons points.} on $k$ for $D_1$, the identities$${z'}^{m_3}=u^{m_1}v^{m_2}$$ and $$w^{m_3}={y'}^{m_2}|x'|^{kn_1}{x'}^{-m_1}$$

\noindent show that, setting $U'_v:=\alpha_{p,v}^{-1}(U_p)$, it induces an isomorphism of real analytic varieties 

$$\overline{\alpha_{p,v}}\colon \left(\{(x',y',z')\in \C^3, x'^{m_1}=y'^{m_2}z'^{m_3}\}\cap U'_{v} \setminus V(y')\right)^N\xrightarrow{\sim} \norm^{-1}\left(\Skt \cap U_p\setminus V(v)\right)$$

We can use this isomorphism to understand the restriction $\overline{\km_{p,v}}$ of $\overline{\km_p}$ to 
$$\left(\{(x,y,z)\in \C^3, x^{m_1}=y^{m_2}z^{m_3}\}\cap \Upk\setminus V(y)\right)^N.$$ 

We use the composition with $\alpha_{p,v}^{-1}$ to obtain a morphism:
$$\alpha_{p,v}^{-1}\circ \km_{p,v} \colon \{(x,y,z)\in \C^3, x^{m_1}=y^{m_2}z^{m_3}\}\cap \Upk \setminus V(y)$$ $$\rightarrow \{(x',y',z')\in \C^3, x'^{m_1}=y'^{m_2}z'^{m_3}\}\cap U'_v \setminus V(y')$$

\noindent given by 

$$\left\{\begin{array}{l}
x'=x\\
y'=y|y|^{-kn_1/m_1}\\
z'=z\cdot |z|^{-kn_1/m_1}|x|^{kn_1/m_3}=z|y|^{\faktor{kn_1m_2}{m_1m_3}}
\end{array} \right. .$$

The morphism $\alpha_{p,v}^{-1}\circ \km_{p,v}$ is not, in general, birational, but it induces an isomorphism of real analytic varieties at the level of normalizations. Indeed, consider the morphism

$$N_1 \colon \{(x_1,y_1,z_1)\in \C^3, y_1^{m_2}z_1^{d_3}=1\} \rightarrow \{(x,y,z)\in \C^3, x^{m_1}=y^{m_2}z^{m_3}\} \setminus V(y)$$

\noindent given by 

$$\left\{\begin{array}{l}
x=x_1^{m_3/d_3}z_1^{a_1}\\
y=y_1\\
z=x_1^{m_1/d_3}z_1^{c_1}
\end{array} \right.$$

\noindent where $d_3=gcd(m_1,m_3)$, and $a_1,c_1\in \N$ are such that $c_1m_3-a_1m_1=d_3$. 

The morphism $N_1$ is a normalization of $\{(x,y,z)\in \C^3, x^{m_1}=y^{m_2}z^{m_3}\} \setminus V(y)$. In the same way, build the normalization 
$$N_1' \colon \{(x_1',y_1',z_1')\in \C^3, y_1'^{m_2}z_1'^{d_3}=1\} \rightarrow \{(x',y',z')\in \C^3, x'^{m_1}=y'^{m_2}z'^{m_3}\} \setminus V(y').$$

Set $U_1:=N_1^{-1}(\Upk),U_1':=N_1'^{-1}(U'_v)$. The morphism $\alpha_{p,v}^{-1}\circ \km_{p,v}$ then lifts to the normalizations in the following way:
$$\overline{\alpha_{p,v}^{-1}\circ \km_{p,1}}\colon \{(x_1,y_1,z_1)\in \C^3, y_1^{m_2}z_1^{d_3}=1\}\cap U_1 \rightarrow \{(x_1',y_1',z_1')\in \C^3, y_1'^{m_2}z_1'^{d_3}=1\}\cap U_1'$$

$$ \left\{\begin{array}{l}
x_1'=x_1|y_1|^{-a_1 \frac{kn_1m_2}{m_1m_3}}\\
y_1'=y_1|y_1|^{-\frac{kn_1}{m_1}}\\
z_1'=z_1|y_1|^{\frac{kn_1m_2}{m_1d_3}}
\end{array} \right.~~~~~~~~~
\left\{\begin{array}{l}
x_1=x_1'|y_1'|^{a_1 \frac{kn_1m_2}{m_3(n_1-km_1)}}\\
y_1=y_1'|y_1'|^{\frac{kn_1}{(m_1-kn_1)}}\\
z_1=z_1'|y_1'|^{-\frac{kn_1m_2}{d_3(m_1-kn_1)}}
\end{array} \right.$$

\noindent which is, as claimed, an isomorphism of real analytic varieties, both functions $|y_1|$ and $|y_1'|$ being everywhere different from $0$.

Now $\overline{\alpha_{p,v}^{-1}\circ \km_{p,v}}$ and $\overline{\alpha_{p,v}^{-1}}$ are diffeomorphisms, whence $\overline{\km_{p,v}}$ also. 

The same strategy will prove that the restriction $\overline{\km_{p,w}}$ of $\overline{\km_p}$ to $$\left(\{(x,y,z)\in \C^3, x^{m_1}=y^{m_2}z^{m_3}\}\cap \Upk \setminus V(z)\right)^N$$ is also an isomorphism of real analytic varieties. 
%
%
%

In conclusion, the images by $\overline{{\km_p}^*}$ of $\norm^{-1}\left(\Skt \cap U_p\cap V(v)\right)$ and of $\norm^{-1}\left(\Skt \cap U_p\cap V(w)\right)$ being disjoint, $\overline{{\km_p}^*}$ is an isomorphism of real analytic varieties.
\end{proof}

\begin{discussion}{\bf Orientation of $\Skt$ and its model near a point $\minus$.}\label{orient moins}

Here again, the orientation of $\Skt$ is given by $du\wedge d\overline{u}\wedge dw\wedge d\overline{w}$. But in order for $\overline{\km_p}$ to be orientation-preserving, we need to orient its source with \textbf{the opposite of the orientation given by its complex structure}, that is, using the form $-dx\wedge d\overline{x}\wedge dy\wedge d\overline{y}$.

Furthermore, the two irreducible components of $\Courbest$ that are visible here, given by $\{u=v=0\}$ and $\{u=w=0\}$, are oriented respectively as complex curves by $dw\wedge d\overline{w}$ and $dv\wedge d\overline{v}$. The morphism $\km_p$ then orients their preimages, respectively $\{x=y=0\}$ and $\{x=z=0\}$, by the opposite of their complex orientations, $-dz\wedge d\overline{z}$ and $-dy\wedge d\overline{y}$.

So, here, the pullback of the orientation of $\Courbest \cap U_p$ by $\norm\circ \km$ \textbf{orients its preimage \linebreak $(\norm\circ \km)^{-1}(\Courbest\cap U_p)$ by the opposite of its complex structure.}
\end{discussion}

\section{A variety $\St$}
In the previous section is implicitly provided a description of an analytic variety $\St$, as well as of a morphism $\K \colon (\St,\Courbesk) \rightarrow (\Sn,\Courbesn)$. In this section we make it more explicit. In particular we explain the orientation of $(\St,\Courbesk)$ and of the surfaces of $\Courbesk$.

\subsection{Definition of $(\St,\Courbesk)$}
Choose a covering of $\Courbes\subset \Skt$ with connected neighbourhoods of the form $U_p\cap \Skt$ providing equations for $\Skt$ as (\ref{eq Skt gen}) and (\ref{eq Skt double}) in subsections \ref{gen pt} and \ref{triple points}. Such a covering may be chosen finite because of the compactness of $\Courbes$. 

In the sequel, if $p\in \Courbes$, $\kappa_p$ denotes $\kappa_p$, $\km_p$ or $\kp_p$, according to the type of the point $p$.

\begin{defn} {\bf (Local complexification.)}\label{complexification}

If $U_p \cap \Skt$ is an open subset in this covering, and if $p$ is a generic point of $\Courbes$ at the intersection of $D_1 \in \Do$ and $D_2 \in \Df$,denote $$\Upt:= \left(\overline{\kappa_p}\circ \norm \right)^{-1}\left(U_p\cap \Skt\right)\subset \{(x_p,y_p,z_p)\in \C^3,x_p^{m_1}=y_p^{m_2}\}^N.$$ In other cases, the definition of $\Upt$ is the same, with the appropriate equation for the right-hand term, according to the type of $p$.

Then we define a real analytic variety $\St$, called \textbf{local complexification of $\Sn$}, by gluing together the open sets $\left(\Upt\right)_p$ in the following way: if $U_p ~\cap ~ U_q \neq \emptyset$, glue $\Upt$ and $\Uqt$ along $$\left(\overline{\kappa_p}\circ \norm \right)^{-1}\left(U_p\cap U_q\cap \Skt\right)\subset \Upt$$ and $$\left(\overline{\kappa_q}\circ \norm \right)^{-1}\left(U_p\cap U_q\cap \Skt\right)\subset \Uqt$$ using $\overline{\kappa_p}^{^{-1}}\circ \overline{\kappa_q}$, which is an isomorphism of real analytic varieties in restriction to $\Upt\cap \Uqt$.
\end{defn}

The latter is indeed an isomorphism: if $p$ is a point of $\Courbes$, and $q\neq p$ a double point of $\Courbest$, we may have $p\in U_q$, but never $q\in U_p$, because of the equations of $\Skt$ on these open sets, see section \ref{section norm loc models}. The isomorphism follows then from Lemmas \ref{diffeo gen}, \ref{diffeo plus} and \ref{diffeo moins}.

\begin{rk}
Note that the open sets $\Upt$ are not, in general, connected. See Section \ref{section configuration complexification} for more details.
\end{rk}

Now, the real analytic morphism $$\K : (\St,\Courbesk) \rightarrow (\Sn,\Courbesn)$$ is defined using the local morphisms $\kappa_p$, on the covering of $\St$ by the $\Upt$'s.

Lemmas \ref{diffeo gen}, \ref{diffeo plus} and \ref{diffeo moins} imply:
\begin{lemma}
The real analytic morphism $\K$ is an isomorphism of real analytic varieties outside $\Courbesn$.
\end{lemma}

\subsection{Orientation of $\St$}\label{orientation St} Although $\St$ is only real analytic, it admits local complex equations, and an orientation of $\St$ can be defined relatively to the local complex orientations, using Discussions \ref{orient gen}, \ref{orient plus} and \ref{orient moins}:
\begin{itemize}
\item On an open $\Upt$ coming from a generic point of $\Courbes$, $\St$ is oriented by its local complex structure, and the preimage of $\Courbes$ also.
\item On an open $\Upt$ coming from a double point $\plus$ of $\Courbest$, $\St$ is oriented by its local complex structure, and the components of the preimage of $\Courbest$ also.
\item On an open $\Upt$ coming from a double point $\minus$ of $\Courbest$, $\St$ is oriented by the opposite of its local complex structure, and the preimage of $\Courbest$ is also oriented by the opposite of its complex structure.
\end{itemize}  

These choices of orientation are compatible, because by construction they ensure that every morphism ${\overline{\kappa_p}}^{-1}\circ \overline{\kappa_q}$ of Definition \ref{complexification} is orientation-preserving. Again, by construction, the choices of orientation for the preimages of the irreducible components of $\Courbest$ by $\norm \circ K$ are compatible (see Discussions \ref{orient gen}, \ref{orient plus} and \ref{orient moins}). Furthermore, as can be checked using those same discussions, 
\begin{lemma}
The morphism $K$ is orientation-preserving.
\end{lemma}

\section{The decorated graph $\Gcdtpart$}\label{section Graphe de courbes}
In order to make the description of $\Courbesk\subset \St$ easier, and in the perspective of the main algorithm, we introduce in this section the decorated graph $\Gcdtpart$.
\begin{defn}
Denote $\Gct$ the dual graph of the configuration of complex curves $\Courbest$. The partially decorated graph $\Gcdtpart$ is obtained from $\Gct$ in the following way:
\hfill
\begin{itemize}
\item If $C$ is an irreducible component of $D_1\cap D_2$, where $D_1\in \Df$ and $D_2\in \Do\cup \Dg$, decorate $v_C$, the vertex corresponding to $C$, with the triple $(m_1;m_2,n_2)$ (with the Notation \ref{multiplicities}), and with its genus $[g]$, in square brackets. If $D_2\in \Dg$, then the vertex associated to the non-compact curve $C$ is an arrowhead.
\item Decorate each edge $e_p$ of $\Gct$ corresponding to a double point $p$ of $\Courbest$ with a $\plus$ or a $\minus$ sign, according to the type of the point $p$, introduced in Definition \ref{def convention signes}. 
\end{itemize}
\end{defn}

\begin{rk}
As in \cite[7.2.1]{NemSzi12}, one can use the fact that the Milnor fiber of $\Phi$ is connected to prove that the graphs $\Gct$ and $\Gc$ are connected. To see that the Milnor fiber of $\Phi$ is connected, one can see it as the generic fiber of the smoothing $g_{|\{g=p\cdot f\}}$ of the reduced space $\Vf \cap \Vg$, for generic values of $p\in \C$, and use \cite[Theorem 8.2 (2)]{Gre17}. 

However, this will not necessarily imply that the plumbing graph obtained at the end of our computation will be connected. Moreover, through move $[R6]$ of the plumbing calculus exposed in \cite{Neu81}, a disconnected graph can encode the same variety as a connected one.
\end{rk}

\begin{rk}
\begin{enumerate}
\item If $p,q\in C_1\cap C_2$, then $p$ and $q$ are of the same type. Hence the different edges between two vertices are all of the same type.
\item Because of the hypothesis of simpleness in the point  \ref{sncd} of Definition \ref{modif adaptee}, the graph $\Gct$ contains no loop.
\end{enumerate}
\end{rk}

\begin{defn}
Let $\ve[C]$ be a vertex of $\Gcdtpart$. Its \textbf{star} is the subgraph of $\Gcdtpart$ whose vertices are $\ve[C]$ and its neighbours, with the edges connecting $\ve[C]$ to them. It will therefore have the following form:

\begin{center}
\begin{tikzpicture}

\filldraw 
(-0.7,0) circle (0pt) node[above] {$(m_1;m_s,n_s)$} node[below] {$[g_s]$}
(6.7,0) circle (0pt) node[above] {$(m_{s+t};m_2,n_2)$} node[below] {$[g_{s+t}]$}
(6,0) circle (2pt) node {}
(0,0) circle (2pt) node {}
(3,0.8) circle (2pt) node {}
(3,1) circle (0pt) node[above] {$(m_1;m_2,n_2)$} 
(3,0.8) circle node[below] {$[g]$}
(0,1.6) circle (2pt) node[above] {$(m_1;m_3,n_3)$} node[below] {$[g_3]$}
(6,1.6) circle (2pt) node[above] {$(m_{s+1};m_2,n_2)$} node[below] {$[g_{s+1}]$}
(0.5,0.9) circle (0pt) node {$\vdots$}
(5.5,0.9) circle (0pt) node {$\vdots$}
;
\draw (0,0) -- node[above]{$\plus$} node[below]{$\mu_s$} (3,0.8);
\draw (0,1.6) -- node[above]{$\plus$} node [below]{$\mu_3$}(3,0.8);
\draw (6,0) -- node[above]{$\minus$} node[below]{$\mu_{s+t}$} (3,0.8);
\draw (6,1.6) -- node[above]{$\minus$} node[below]{$\mu_{s+1}$}(3,0.8);

\end{tikzpicture}

\end{center}

\noindent where the decoration $\mu_i$ means that the edge is repeated $\mu_i$ times.
\end{defn}

\section{Description of $\Gkdt$}\label{section configuration complexification}

On a local chart of $\St$ of the form $\Upt$ (see Definition \ref{complexification}), we know how to recognize the configuration $\Courbesk$. However, we also need to know it as a global object. Namely, two questions remain:
\hfill
\begin{enumerate}
\item Given an irreducible curve $C\in \Courbes$, what is the structure of $\Ck$?
\item Given two irreducible curves $C_1,C_2\in \Courbes$ such that $C_1\cap C_2 \neq \emptyset$, what can we say about $\Cuk \cap \Cdk$?
\end{enumerate}

We know that in every open set $\Upt$, $\Courbesk$ is a configuration of complex curves.

In the sequel, $C$ is an irreducible component of $\Courbes$, contained in $D_1 \cap D_2$, where $D_1\in \Df$ and $D_2\in \Dg$ or $\Do$. Denote $p_3,\cdots,p_l$ the double points of $\Courbes$ on $C$, and $$d:=gcd(m_1,m_2).$$
\subsection{Cyclic orders on components}
In this subsection, we show how one can endow coherently the set of local irreducible components of $\Skt$ at a generic point of $C$ with a cyclic order. Then we show that the same is true for double points on $C$, and that the cyclic orders defined are compatible.

\begin{lemma}\label{ordre cyclique}{\bf (Cyclic order at generic points.)}

If $p$ is a generic point of the irreducible component $C$ of $D_1\cap D_2$, $\Ck\cap \Upt$ is a disjoint union of $d$ complex curves. Furthermore there is a cyclic order (in the sense of Definition \ref{def ordre cyclique}) on them. This order is compatible with the gluings of charts $\Uqt$ coming from generic points $q\in \Courbest$.

Namely, if $\St\cap \Upt=\left(\left\{x_p^{m_1}=y_p^{m_2}\right\}\cap \Upk\right)^N$, then the decomposition in irreducible components $$\left\{x_p^{m_1}=y_p^{m_2}\right\}=\bigcup\limits_{j=0}^{d-1} \left\{x_p^{m_1/d}=e^{\frac{2ij\pi}{d}} y_p^{m_2/d}\right\},$$ where $d$ denotes $gcd(m_1,m_2)$, provides $\St\cap \Upt$ as a disjoint union 
$$\St\cap \Upt=\bigsqcup\limits_{j=0}^{d-1} \left(\left\{x_p^{m_1/d}=e^{\frac{2ij\pi}{d}} y_p^{m_2/d}\right\}\right)^N\cap \Upt$$ 
and finally 
$$\Ck\cap \Upt=\bigsqcup\limits_{j=0}^{d-1} C^{(j)}\cap \Upt$$ where $C^{(j)}$ is the preimage of the curve $\{x_p=y_p=0\}$ by the normalization of the irreducible component $\{x_p^{m_1/d}=e^{\frac{2ij\pi}{d}} y_p^{m_2/d}\}$. Each $C^{(j)}$ is a smooth complex curve. The cyclic order on the connected components of $\Ck$ is then given by the rule: $$C^{(j+1)} \text{ follows } C^{(j)}\text{, where the addition is to be understood in }\Z/d\Z.$$ or equivalently $$\left\{x_p^{m_1/d}=e^{\frac{2i(j+1)\pi}{d}} y_p^{m_2/d}\right\}\cap \Upk \text{ follows } \left\{x_p^{m_1/d}=e^{\frac{2ij\pi}{d}} y_p^{m_2/d}\right\}\cap \Upk$$
\end{lemma}
\begin{proof}
We want to prove that, given two open subsets $\Upt$ and $\Uqt$, coming from generic points $p,q\in C$, such that $\Uqt\cap \Upt\neq \emptyset$, the orderings on the components of $\Ck$ on $\Upt\cap \Uqt$ are the same. The point is that this ordering on curves is the same as an ordering on irreducible components.

First, note that we can relate the coordinates $(x_p,y_p,z_p)$ and $(x_q,y_q,z_q)$ by descending to $\Skt$: take coordinates $(u_p,v_p,w_p)$ on $\Skt \cap U_p$ and $(u_q,v_q,w_q)$ on $\Skt \cap U_q$, as in section \ref{gen pt}. Then there are functions $\lambda,\mu,\nu$, nowhere zero on $U_p\cap U_q$ such that, on $U_p\cap U_q$,

$$\left\{ \begin{array}{l}
u_q=u_p \cdot \lambda(u_p,v_p,w_p)\\ 
v_q=v_p \cdot \mu(u_p,v_p,w_p)\\
w_q=\nu(u_p,v_p,w_p)
\end{array} \right.$$
which leads, using equations (\ref{morph gen}) of subsection \ref{gen pt}, to the following identities on $\Upk\cap \Uqk$:

$$\left\{ \begin{array}{l}
x_q=x_p \cdot \lambda\\ 
y_q=y_p \cdot\mu^{-1}|\lambda|^{kn_1/m_2}\\
z_q=\nu
\end{array} \right.$$
where $\lambda$ denotes $\lambda(x_p,y_p^{-1}|x_p|^{kn_1/m_2},z_p)$, and similarily for $\mu, \nu$. The functions $\lambda,\mu^{-1}|\lambda^{kn_1/m_2}|,\nu$ are nowhere zero on $\Upk\cap \Uqk$.

Whence the identification $$\left\{x_q^{m_1/d}=e^{\frac{2ij\pi}{d}} y_q^{m_2/d}\right\}\cap \Upk\cap \Uqk= \left\{x_p^{m_1/d}\lambda^{m_1/d}=e^{2ij\pi/d}y_p^{m_2/d}\mu^{-m_2/d}|\lambda|^{kn_1/d}\right\}\cap \Upk\cap \Uqk.$$

This implies that there exists $0\leqslant l \leqslant j-1$ such that $$e^{2ij\pi/d}\mu^{-m_2/d}|\lambda|^{kn_1/d}\lambda^{-m_1/d}=e^{2il\pi/d}.$$

This allows one to identify the $j$-th irreducible component, as seen in $\Uqk$, with the $l$-th component, in $\Upk$. And replacing $j$ by $j+1$ in the previous computation leads now to the component $l+1$ in $\Uqk$. 

In other words, we gave a meaning to the expression ``the following irreducible component'' at a point of $\{x_p^{m_1}=y_p^{m_2}\}\cap \Upk$, and hence to the expression ``the following curve'' in $\St\cap \Upt$.
\end{proof}

\begin{discussion}{\bf (Cyclic order at double points.)}\label{discussion ordre cyclique points doubles}
Recall the notations introduced at he beginnning of Section \ref{section configuration complexification}. Let $p_i$ be a double point of $\Courbest$ in $C$, $p_i\in D_i$. Then $\NK^{-1}(p_i)$ is made of $d_i:=gcd(m_1,m_2,m_i)$ points $\overline{p_i^1},\cdots,\overline{p_i^{d_i}}$, each one of them corresponding to an irreducible component of $\kappa_{p_i}^{-1}(\Skt\cap U_{p_i})$. 

Furthermore, one can define a cyclic order on these irreducible components, in the following way:

\textbf{If $p_i$ is of type $\plus$}, then \begin{equation}\kappa_{p_i}^{-1}(\Skt\cap U_{p_i})=\label{eq avant norm point plus}
\bigcup\limits_{j=0}^{d_i-1} \left\{y_i^{m_2/d_i}=e^{2ij\pi/d_i} x_i^{m_1/d_i}z_i^{m_3/d_i}\right\}\cap \Upik 
\end{equation}
where the coordinates $(x_i,y_i,z_i)$ are defined from local coordinates $(u_i,v_i,w_i)$ at $p_i$ \emph{via} Equation \ref{eq morph plus} applied at $p_i$. This union is the decomposition of $\kappa_{p_i}^{-1}(\Skt\cap U_{p_i})$ in irreducible components.

We define a cyclic order on the set of irreducible components by the following rule: $$\left\{y_i^{m_2/d_i}=e^{2ij\pi/d_i} x_i^{m_1/d_i}z_i^{m_3/d_i}\right\}$$ is followed by $$\left\{y_i^{m_2/d_i}=e^{2i(j-1)\pi/d_i} x_i^{m_1/d_i}z_i^{m_3/d_i}\right\}.$$

\textbf{If $p_i$ is of type $\minus$}, then \begin{equation}
\kappa_{p_i}^{-1}(\Skt\cap U_{p_i}) \label{eq avant norm point moins}=
\bigcup\limits_{j=0}^{d_i-1} \left\{x_i^{m_1/d_i}=e^{2ij\pi/d_i} y_i^{m_2/d_i}z_i^{m_3/d_i}\right\}\cap \Upik 
\end{equation}
where the coordinates $(x_i,y_i,z_i)$ are defined from local coordinates $(u_i,v_i,w_i)$ at $p_i$ \emph{via} Equation \ref{eq morph minus} applied at $p_i$.

Then, again, define a cyclic order by the rule: the component $$\left\{x_i^{m_1/d_i}=e^{2ij\pi/d_i} y_i^{m_2/d_i}z_i^{m_3/d_i}\right\}$$ is followed by the component $$\left\{x_i^{m_1/d_i}=e^{2i(j+1)\pi/d_i} y_i^{m_2/d_i}z_i^{m_3/d_i}\right\}.$$
\end{discussion}

The normalization of complex-analytic varieties having the topological effect of separating local irreducible components, this means that $\St\cap \Upit$ is made of a disjoint union of $d_i$ normal varieties, with a cyclic order induced by the one on the irreducible components of $\kappa_{p_i}^{-1}(\Skt\cap U_{p_i})$. Incidentally, this provides also a cyclic order on the set $\left\{\overline{p_i^1},\cdots,\overline{p_i^{d_i}}\right\}$ of preimages of $p_i$ by $\NK$. 

Moreover, the intersection of $\Ck$ with each of these connected components is an irreducible curve, namely the pullback of $\{x_i=y_i=0\}$ by the normalization of the irreducible component, with the notations of Equations (\ref{eq avant norm point plus}) or (\ref{eq avant norm point moins}). See Subsection \ref{appendix resol HJ} for more details. Recall again the setting of the beginning of Section \ref{section configuration complexification}. Then:

\begin{lemma}\label{regroupement pts doubles}{\bf (Compatibility of orders.)}

Let $p_i\in C \cap D_i$ be one of the double points of $\Courbest$ on $C$. Let $p$ be a generic point of $C$ such that $\widetilde{U}:=\Upt\cap \Upit\neq \emptyset$. Consider the cyclic orders defined above. If an irreducible component of $\kappa_p^{-1}(\Skt\cap \overline{U})$ is contained in an irreducible component of $\kappa_{p_i}^{-1}(\Skt\cap \Upit)$, then the following component of $\kappa_p^{-1}(\Skt\cap \overline{U})$ is contained in the following irreducible component of $\kappa_{p_i}^{-1}(\Skt\cap \Upit)$.

In particular, two curves in $\Ck\cap \overline{U}$ are in the same branch of $\Ck \cap \Upit$ if and only if their positions in the cyclic order differ by a multiple of $d_i$.

\end{lemma}

\begin{proof}[Proof of Lemma \ref{regroupement pts doubles}.]
Denote $(u,v,w)$ and $(u_i,v_i,w_i)$ coordinates on $U_p$ and $U_{p_i}$ giving respectively equations of the form (\ref{eq Skt gen}) and (\ref{eq Skt double}) for $\Skt$, as in Section \ref{section norm loc models}. 

Then $$\St\cap  \Upt=\left(\left\{x^{m_1}=y^{m_2}\right\}\cap \Upk\right)^N=\bigsqcup\limits_{j=0}^{d-1} \left(\left\{x^{m_1/d}=e^{2ij\pi/d}y^{m_2/d} \right\}\cap \Upk\right)^N$$
where $$\left\{\begin{array}{l}
x=u\\
y=v^{-1}|u|^{kn_1/m_2}\\
z=w
\end{array} \right.$$

Then there exist functions $\lambda,\mu,\nu$, nowhere zero on $U_p\cap U_{p_i}$, such that, on $U_p\cap U_{p_i}$,
$$\left\{ \begin{array}{l}
u=u_i . \lambda(u_i,v_i,w_i)\\ 
v=v_i .\mu(u_i,v_i,w_i)\\
w=\nu(u_i,v_i,w_i)
\end{array} \right.$$

We face two different situations, according to the type of the double point $p_i$. Denote $$\delta_i:=\dfrac{d}{d_i}.$$

\underline{If $p_i$ is of type $\plus$}, then $$\St\cap  \Upit=\left(\left\{y_i^{m_2}=x_i^{m_1}z_i^{m_3}\right\}\cap \Upik\right)^N=$$ \begin{equation}\label{eq multigerme point plus}
\bigsqcup\limits_{j=0}^{d_i-1} \left(\left\{y_i^{m_2/d_i}=e^{2ij\pi/d_i} x_i^{m_1/d_i}z_i^{m_3/d_i}\right\}\cap \Upik\right)^N
\end{equation}
where 
$$\left\{\begin{array}{l}
u_i=x_i\\
v_i=y_i^{-1}|x_i|^{kn_1/m_2}|z_i|^{kn_3/m_2} \\
w_i=z_i
\end{array} \right.$$
which gives, on $\Upik\cap \Upk$,
$$\left\{\begin{array}{l}
x=x_i\lambda\\
y=y_i|z_i|^{-kn_3/m_2}\mu^{-1}|\lambda|^{kn1/m_2}\\
z=\nu
\end{array} \right.$$
where $\lambda$ (or $\mu,\nu$) denotes $\lambda(x_i,y_i^{-1}|x_i|^{kn_1/m_2}|z_i|^{kn_3/m_2},z_i)$. Hence 
$$\left\{x^{m_1/d}=e^{2ij\pi/d} y^{m_2/d}\right\}\cap \Upk\cap \Upik=$$ $$\left\{x_i^{m_1/d}=y_i^{m_2/d}\lambda^{-m_1/d}|\lambda|^{kn_1/d}|z_i|^{-kn_3/d}\mu^{-m_2/d}e^{2ij\pi/d}\right\}\cap \Upk\cap \Upik.$$

Elevating to the power $\delta_i$ shows that $$\left\{x^{m_1/d}=e^{2ij\pi/d} y^{m_2/d}\right\}\cap \Upk\cap \Upik\subset$$ $$\left\{y_i^{m_2/d_i}=x_i^{m_1/d_i}\lambda^{m_1/d_i}|\lambda|^{-kn_1/d_i}|z_i|^{kn_3/d_i}\mu^{m_2/d_i}e^{-2ij\pi/d_i}\right\}\cap \Upk\cap \Upik.$$

Now there exists $0\leqslant l\leqslant d_i-1$ such that on $\Upik\cap \Upk$, $$\lambda^{m_1/d_i}|\lambda|^{-kn_1/d_i}|z_i|^{kn_3/d_i}\mu^{m_2/d_i}e^{-2ij\pi/d_i}=e^{2il\pi/d_i} z_i^{m_3/d_i}.$$

The previous inclusion can then be written as 
$$\left\{x^{m_1/d}=e^{2ij\pi/d} y^{m_2/d}\right\}\cap \Upk\cap \Upik\subset \left\{y_i^{m_2/d_i}=e^{2il\pi/d_i}x_i^{m_1/d_i} z_i^{m_3/d_i}\right\}\cap \Upk\cap \Upik.$$

Taking the successor of this component, we get indeed $$\left\{x^{m_1/d}=e^{2i(j+1)\pi/d} y^{m_2/d}\right\}\cap \Upk\cap \Upik \subset \left\{y_i^{m_2/d_i}=e^{2i(l-1)\pi/d_i}x_i^{m_1/d_i} z_i^{m_3/d_i}\right\}\cap \Upk\cap \Upik.$$

\underline{If $p_i$ is of type $\minus$,}  then $$\St\cap  \Upit=\left(\left\{x_i^{m_1}=y_i^{m_2}z_i^{m_3}\right\}\cap \Upik\right)^N=$$ \begin{equation}\label{eq multigerme point moins}
\bigsqcup\limits_{j=0}^{d_i-1} \left(\left\{x_i^{m_1/d_i}=e^{2ij\pi/d_i} y_i^{m_2/d_i}z_i^{m_3/d_i}\right\}\cap \Upik\right)^N
\end{equation}

and the proof follows the same steps.

\end{proof}
\begin{rk}
In other words, taking into account a double point $p_i$ of $C$  shows that some of the \emph{a priori} different components of $\Ck$ are in the same branch, and that this identification is made uniformly with respect to the cyclic order. 
\end{rk}
However, other identifications may occur, due to the fundamental group of $C$. We make this aspect more precise in Subsection \ref{subsection non rational curves}.

The following lemma describes a favorable situation in which the data contained in $\Gcdtpart$ is sufficient to determine the structure of $\Ck$. To state it we need the following:

\begin{defn}
Let $(L,0)$ be a branch of $\Sigma \Vf$. 
The \textbf{transversal type} of $(X,\Vf)$ along $L$ is the equisingularity type of a transverse section of the couple $(X,\Vf)$ along $L$, where two curves in surfaces $(V_1,C_1)$ and $(V_2,C_2)$ are called equisingular if they admit homeomorphic good embedded resolutions.
\end{defn}
\begin{rk}
This definition makes sense because there exists a representative $X$ of $(X,0)$ such that for any branch $L$ of $\Sigma \Vf$, all transverse sections along $L$ are equisingular.
\end{rk}
\begin{lemma}\label{Rational curve}
Let $C$ be an ireducible curve of $\Courbes$, and denote $D$ the component of $\Df$ containing $C$. If $D$ is not contained in the strict transform $\Vft$ of $\Vf$, denote by $L$ the curve $\rx(D)$. If the transversal type of $(X,\Vf)$ along $L$ has a resolution whose exceptional divisor is a union of rational curves, then the curve $C$ is rational.
\end{lemma}

The proof of this lemma follows the one of \cite[Proposition 7.4.8, a)]{NemSzi12}. In the case studied by Némethi and Szil\' ard, it has the following consequence:
\begin{cor}
Let $(X,0)=(\C^3,0)$, and $C\in \Courbes$ be decorated by a triple of multiplicities $(m_1;m_2,n_2)$. Then if $m_1\geqslant 2$, then the curve $C$ is rational.
\end{cor}

Indeed, in this case all transversal types of the surface $\Vf$ are germs of plane curves, whose resolutions use only rational curves.

\subsection{Non-rational curves}\label{subsection non rational curves}
However, in general, the curve $C$ may not be rational, whence the need of a little additional information, namely about the identifications of local irreducible components of $\Skt$ along loops in $C$. The goal of this subsection is to make this precise.

Let $C$ be an irreducible curve of $\Courbes$. Again, let $p_3,\dots,p_l$ be the double points of $\Courbes$ on $C$. Let $p\in C\setminus \{p_3,\dots,p_l\}$. Then $\pi_1(C\setminus \{p_3,\dots,p_l\},p)$ acts on the set of local irreducible components of $\Skt$ at $p$ in the following way: let $[\gamma]\in \pi_1(C\setminus \{p_3,\dots,p_l\},p)$. Consider a tubular neighbourhood $U=\bigcup\limits_{i=0}^{j} U_i$ of $\gamma$ in $\Xt$ such that, on each $U_i$, we can have local coordinates in $\Xt$ in such a way as to get local equations as in \ref{gen pt} for $\Skt$. Then one can follow an irreducible component at $p$ along $\gamma$. By consistency of the cyclic order, this action is completely defined by the data of the number $\delta$ such that a component is sent on the one which is $\delta$ further in the cyclic order. Furthermore it is clear that:

\begin{lemma} 
For $[\gamma]$ in $\pi_1(C\setminus \{p_3,\dots,p_l\})$, the number $\delta$ is independent of the choice of the base point $p$ or the representative $\gamma$. $\square$
\end{lemma}

\begin{defn}\label{def switch}
The number $\delta$ defined above is called the \textbf{switch} associated to $\gamma$.
\end{defn}

Now, $\pi_1(C\setminus \{p_3,\dots,p_l\},p)$ can be generated by $\alpha_1,\dots,\alpha_{2g},\gamma_1,\dots,\gamma_l$, where $\alpha_1,\dots,\alpha_{2g}$ generate $\pi_1(C,p)$, and $\gamma_i$ is a ``simple loop'' around $p_i$, i.e. $\gamma_i\neq 1$ in $\pi_1(C\setminus \{p_1,\dots,p_l\},p)$ but $\gamma_i=1$ in $\pi_1(C\setminus \{p_1,\dots,p_{i-1},p_{i+1},\dots,p_l\},p)$

Lemmas \ref{ordre cyclique} and \ref{regroupement pts doubles}, together with the previous considerations, imply the following:

\begin{thm}\label{thm nombre courbes normalisation}{\bf (Structure of $\Ck$.)}

Let $C$ be a curve of $\Courbes$ of genus $g$, and let its star in $\Gcdtpart$ be
\begin{center}
\begin{tikzpicture}

\filldraw 
(-0.7,0) circle (0pt) node[above] {$(m_1;m_s,n_s)$} node[below] {$[g_s]$}
(6.7,0) circle (0pt) node[above] {$(m_{s+t};m_2,n_2)$} node[below] {$[g_{s+t}]$}
(6,0) circle (2pt) node {}
(0,0) circle (2pt) node {}
(3,0.8) circle (2pt) node {}
(3,1) circle (0pt) node[above] {$(m_1;m_2,n_2)$} 
(3,0.8) circle node[below] {$[g]$}
(0,1.6) circle (2pt) node[above] {$(m_1;m_3,n_3)$} node[below] {$[g_3]$}
(6,1.6) circle (2pt) node[above] {$(m_{s+1};m_2,n_2)$} node[below] {$[g_{s+1}]$}
(0.5,0.9) circle (0pt) node {$\vdots$}
(5.5,0.9) circle (0pt) node {$\vdots$}
;
\draw (0,0) -- node[above]{$\plus$} node[below]{$\mu_s$} (3,0.8);
\draw (0,1.6) -- node[above]{$\plus$} node [below]{$\mu_3$}(3,0.8);
\draw (6,0) -- node[above]{$\minus$} node[below]{$\mu_{s+t}$} (3,0.8);
\draw (6,1.6) -- node[above]{$\minus$} node[below]{$\mu_{s+1}$}(3,0.8);

\end{tikzpicture}

\end{center}

Let $\delta_i$ be the switch associated to $\alpha_i$, where $\alpha_1,\dots,\alpha_{2g}$ generate $\pi_1(C)$. Then $\Ck$ is the union of $$n_C:=gcd(m_1,\dots,m_{s+t},\delta_1,\dots,\delta_{2g})$$ disjoint connected irreducible real analytic oriented surfaces.
The set of connected components of $\Ck$ is endowed with a cyclic order induced by the one on the local irreducible components of $\kappa_p^{-1}(\Skt\cap U_p)$ at generic points $p$ of $C$. Denote $\overline{C_1},\cdots,\overline{C_{n_C}}$ those surfaces, where the numbering respects the cyclic order.

Furthermore, the common Euler characteristic $\chi\left(\overline{C}\right)$ of each of these surfaces verifies \begin{equation} \label{eq calcul carac euler} n_C\cdot \chi\left(\overline{C}\right)=\left(2-2g-\sum\limits_{i=3}^{s+t}\mu_i\right)\cdot gcd(m_1,m_2) + \sum\limits_{i=3}^{s+t}gcd(m_1,m_2,m_i)\cdot \mu_i.
\end{equation}

\end{thm} 
\begin{proof}[Proof of the last identity] This identity comes by  recognising $gcd(m_1,m_2)$ as the cardinal of the preimage of a generic point, and $gcd(m_1,m_2,m_i)$ as the cardinal of the preimage of a double point, and applying the Riemann-Hurwitz formula for the map $\norm\circ K$ restricted to the irreducible curve $C$.
\end{proof}

\begin{rk}
The collection of switches associated to the $\alpha_i$'s is of course not unique, but the number $n_C$ is. 
\end{rk}

\subsection{Preimage of an intersection point}
Now, the final information we need to define the graph $\Gn$ is the data above the double points of $\Courbes$. Let $C,C'$ be irreducible curves in $\Courbes$, and $p\in C\cap C'$, $p\in D_1\cap D_2\cap D_3$. Denote $d:=gcd(m_1,m_2,m_3)$. Denote $\overline{p^1},\cdots,\overline{p^d}$ the ordered preimages of $p$ by $\NK$ (see Discussion \ref{discussion ordre cyclique points doubles}). Recall the notations of Theorem \ref{thm nombre courbes normalisation}. Lemma \ref{regroupement pts doubles} implies the following:
 
\begin{lemma} \label{fact intersection.}
for any $\overline{p^i}\in \NK^{-1}(p)$, there are components $\overline{C_j},\overline{C_l'}$, respectively of \linebreak $\Ck$ and $\NK^*(C')$, such that $\overline{p^i}\in \overline{C_j}\cap \overline{C'_l}$. 

Furthermore, if $\overline{p^l}\in \overline{C_i}\cap \overline{C'_j}$, then $\overline{p^{l+1}}\in \overline{C_{i+1}}\cap \overline{C'_{j+1}}$.
\end{lemma}

\subsection{The graph $\Gkdt$}\label{subsection graphe gkdt}In this subsection, we sum up what has been developped throughout the previous ones to describe the decorated graph $\Gkdt$.

If $p$ is a double point of $\Courbest$, $p\in D_1\cap D_2\cap D_3$, denote $d_p=gcd(m_1,m_2,m_3)$ the cardinal of $\NK^{-1}(p)$. 

Recall the notations of Section \ref{section Graphe de courbes} and Subsection \ref{subs graph cover}. Theorem \ref{thm nombre courbes normalisation} and Lemma \ref{fact intersection.} imply:
\begin{lemma}\index{Cyclic covering}
The graph $\Gkt$ is a cyclic covering of $\Gct$ with covering data $\left(\{n_C\}_{C\in \Courbest},\{d_p\}_{p\text{ double point}}\right)$.
\end{lemma}

However, in general, this data does not determine $\Gkt$ uniquely. See \cite{Nem00} for cases where this data is enough. In particular, if $(X,0)=(\C^3,0)$, $\pazocal{G}\left(\{n_C\}_{C\in \Courbest},\{d_p\}_{p\text{ double point}}\right)=0$, hence the covering data is enough, see \cite[Theorem 7.4.16]{NemSzi12}. 

If $v\in \PV(\Gct)$ is associated to a curve $C\in \Courbest$, denote $n_v:=n_C$, and denote $\overline{v_1},\cdots,\overline{v_{n_v}}$ the vertices of $\Gkt$ associated to $v$ in the covering. In the same fashion, if $e\in \PE(\Gct)$ is associated to a double point $p\in \Courbest$, denote $d_e:=d_p$, and denote $\overline{e_1},\cdots,\overline{e_{d_e}}$ the edges associated to $e$.

To help overcome the ambiguity, note that the structure of cyclic covering implies that if one knows the end-vertices of one of the edges $\overline{e_i}$ for each $e\in \PE(\Gct)$, then one knows $\Gkt$, up to isomorphism of cyclic coverings. It is therefore enough, for every double point $p\in C\cap C'$ of $\Courbest$, to figure out two components $\overline{C_i}$ and $\overline{C'_l}$ intersecting at one of the $\overline{p^j}$'s.



\begin{defn}{\bf (Decorations of $\Gkdt$.)}

The decorated graph $\Gkdt$ is obtained by decorating every vertex and every edge of $\Gkt$ with the decorations of its image in $\Gcdtpart$, \textbf{except for the genus decorations}: if $v\in \PV(\Gct)$ corresponds to a curve $C$, for each $\overline{v_i}$ corresponding to $v$, replace the genus decoration by the Euler characteristic $\chi\left(\overline{C}\right)$ of $\overline{C_i}$ computed in Theorem \ref{thm nombre courbes normalisation}.
\end{defn}

\section{The resolution step and a variety $(\Sb,E_t)$}\label{resolution step}

\subsection{A variety $(\Sb,E_t)$}

The final step of the construction of the morphism $\Pi$ announced in Section \ref{section tower} is the resolution $\pi$ of the remaining isolated singular points of $\St$. For each singular point $\pt$ of $\St$, we are going to decribe a resolution $\pi_\pt$ of the singularity $(\St,\pt)$, in the following subsections \ref{subsection resol plus} and \ref{subsection resol moins}.

Once this process is complete, denote $(\Sb,E_t)$ the variety obtained from $(\St,\Courbeskt)$ by resolving its singular points, and $$\pi\colon (\Sb,E_t) \rightarrow (\St,\Courbeskt)$$ the global resolution of $(\St,\Courbeskt)$ obtained by gluing the $\pi_\pt$'s.

Finally, denote $$\Pi:=\tr\circ \norm\circ K \circ\pi\colon (\Sb,E)\rightarrow (\Sk,0)$$ the composed morphism, which is a resolution of the $4$-dimensional real analytic singularity $(\Sk,0)$.

Note that the construction of $\Pi$ ensures that it is an isomorphism of real analytic varieties outside the origin.

Let $\pt\in \Ct\cap \Cpt$, where $\Ct,\Cpt$ are irreducible real surfaces of $\Courbeskt$, be a potentially singular point of $\St$. We face two different situations, according to the type of the point $$p:=(\norm\circ K)(\pt)\in D_1\cap D_2\cap D_3,$$ where $D_1\in \Do \cup \Dg, D_2\in \Df$ and $D_3 \in \Dt$. Denote $$C:=(\norm\circ K)\left(\Ct\right),C':=(\norm\circ K)\left(\Cpt\right).$$ Denote also $\Ut$ the connected component of $\Upt$ containing $\pt$, and $d:=(m_1,m_2,m_3)$.

\subsection{Point of type $\plus$}\label{subsection resol plus}

If $p$ is of type $\plus$, then the edge associated to $\pt$ in $\Gkdt$ is of the form 
\begin{center}
\begin{tikzpicture}

\filldraw 
(0,0.3) circle (0pt) node[above] {$(m_2;m_1,n_1)$} 
(3,0.3) circle (0pt) node[above] {$(m_2;m_3,n_3)$}
(0,0) circle (2pt) node[below] {} 
(3,0) circle (2pt) node[below] {}
(0,-0.2) circle (0pt) node[below] {$[\chi]$} 
(3,-0.2) circle (0pt) node[below] {$[\chi']$}
(-0.2,0) circle (0pt) node[left] {$v_{\Ct}$} 
(3.2,0) circle (0pt) node[right] {$v_{\Cpt}$} 
;
\draw (0,0) -- node[above]{$\plus$} (3,0);

\end{tikzpicture}
\end{center}

\noindent and, on $\Ut$, the equation of $\St$ is of the form 
$$\St\cap \Ut=\left(\left\{y^{m_2/d}=e^{2ij\pi/d} x^{m_1/d}z^{m_3/d}\right\}\cap \Upk\right)^N.$$
(see equation (\ref{eq multigerme point plus}) in the proof of Lemma \ref{regroupement pts doubles}.)

Denote $$\pi_\pt \colon \Uti \rightarrow \St \cap \Ut$$ the resolution of the surface $\left(\left\{x^{m_1/d}=e^{2ij\pi/d} y^{m_2/d}z^{m_3/d}\right\}\right)^N\cap \Ut$ described in Appendix \ref{appendix resol HJ}, up to exchanging the roles of the coordinates $x$ and $y$. In the sequel we refer to this Appendix for the notations.

Note that  $$\Ct\cap \Ut=\left(norm^*\left(\{y=x=0\}\right)\right)\cap \Ut$$
and
$$\Cpt\cap \Ut=\left(norm^*\left(\{y=z=0\}\right)\right)\cap \Ut.$$
Denote $\Cti$, $\Cpti$ the strict transforms of $\Ct,\Cpt$ by the morphism $\pi$.

Denote $$\gk:=g \circ \tr \circ \norm \circ \kappa_p$$ the pullback of the function $g$ to $\{y^{m_2}=x^{m_1}z^{m_3}\}\cap \Upk$. Up to a unit, $$\gk=x^{n_1}z^{n_3}.$$

Under the morphism $\pi$, the preimage of $$\Courbeskt\cap \Ut=norm^* \left(\left\{xz=y=0\right\}\right)\cap \Ut$$ is a chain of complex curves, namely $$Str\left(\frac{m_2}{d};\frac{m_1}{d},\frac{m_3}{d}|0;n_1,n_3\right)$$ where the notation $Str$ is explained in Appendix \ref{appendix resol HJ}. The multiplicities on the curves of the string are those of the function $$\gti:=g\circ \Pi$$ which is the pullback of the function $g$ to $\Sb\cap \pi^{-1}(\Ut)$, and where $\widetilde{C_0}=\Cpti$, and $\widetilde{C_{l+1}}=\Cti$.

\noindent Denote $\Uti:=\pi^{-1}\left(\Ut\right)$. Note that the morphism $\pi\colon \Uti \rightarrow \Ut$ is an isomorphism outside of $\pt$.

\begin{discussion}{\bf Orientation compatibilities, point $\plus$.}

Let us remind here the choices of orientation made in subsection \ref{orientation St}. 

The pullbacks of the orientations of the curves $\Ct\cap \Ut,\Cpt\cap \Ut$ by the biholomorphism $\pi$ are again the orientations given by their local complex structure. 

Each new curve $\widetilde{C_i}$ is taken oriented by \textbf{its complex structure}. 

The pullback of the orientation of $\St\cap \Ut$ by the bimeromorphism $\pi\colon \Uti \rightarrow \Ut$ is, again, the orientation given by the complex structure. Now, the compatibility of the complex orientations implies that at each intersection point $\widetilde{C_i}\cap \widetilde{C_{i+1}}$, the combination of the orientations of the two curves gives the ambient orientation of $\Uti$.
\end{discussion}

\begin{rk}\label{strict transform OK}
It is useful in practice to notice that, if $m_2=1$, the point $\pt$ is already a smooth point of $\St$. The morphism $\pi$ described in Appendix \ref{appendix resol HJ} is in this case an isomorphism.
\end{rk}

In the end, the preimage of $(\Courbeskt,\pt)$ by $\pi$ is a chain of complex curves, whose dual decorated graph is the bamboo of figure \ref{fig:chaineplus}, where the $\plus$ signs refer to the orientation compatibilities. The integers $\mu_i$ represent the multiplicities of $\gti$ on each curve. For the definition of the multiplicity of $\gti$ on the real surfaces $\Cti$ and $\Cpti$, we refer to Definition \ref{def multiplicity}.

\begin{figure}[h]
\begin{center}

        \includegraphics[totalheight=1.1cm]{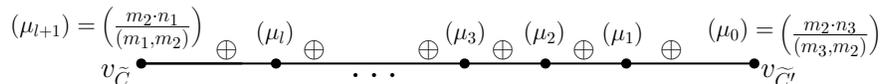}
    \caption{The string $Str^\plus\left(\frac{m_2}{d};\frac{m_1}{d},\frac{m_3}{d}|0;n_1,n_3\right)$.}
    \label{fig:chaineplus}
\end{center}
\end{figure}

Recall that $(a,b)$ denotes $gcd(a,b)$.

\subsection{Point of type $\minus$}\label{subsection resol moins}
If $p$ is of type $\minus$, then the edge associated to $\pt$ in $\Gkdt$ is of the form 
\begin{center}
\begin{tikzpicture}

\filldraw 
(0,0.3) circle (0pt) node[above] {$(m_2;m_1,n_1)$} 
(3,0.3) circle (0pt) node[above] {$(m_3;m_1,n_1)$}
(0,0) circle (2pt) node[below] {} 
(3,0) circle (2pt) node[below] {}
(0,-0.2) circle (0pt) node[below] {$[\chi]$} 
(3,-0.2) circle (0pt) node[below] {$[\chi']$}
(-0.2,0) circle (0pt) node[left] {$v_{\Ct}$} 
(3.2,0) circle (0pt) node[right] {$v_{\Cpt}$} 
;
\draw (0,0) -- node[above]{$\minus$} (3,0);

\end{tikzpicture}
\end{center}

\noindent and, on $\Ut$, the equation of $\St$ is of the form 
$$\St\cap \Ut=\left(\left\{ x^{m_1/d}=e^{2ij\pi/d}y^{m_2/d}z^{m_3/d}\right\}\cap \Upk\right)^N.$$

See equation (\ref{eq multigerme point moins}) in the proof of Lemma \ref{regroupement pts doubles}.

Denote $$\pi_\pt \colon \Uti \rightarrow \St \cap \Upt$$ the resolution of the surface $\left(\left\{x^{m_1/d}=e^{2ij\pi/d} y^{m_2/d}z^{m_3/d}\right\}\right)^N\cap \Ut$ described in Appendix \ref{appendix resol HJ}. Again, in the sequel we refer to this section for the notations.

Note that  $$\Ct\cap \Ut=\left(norm^*\left(\{x=z=0\}\right)\right)\cap \Ut$$
and
$$\Cpt\cap \Ut=\left(norm^*\left(\{x=y=0\}\right)\right)\cap \Ut.$$
Denote $\Cti$, $\Cpti$ the strict transforms of $\Ct,\Cpt$ by the morphism $\pi$.

Denote $$\gk:=g \circ \tr \circ \norm \circ \kappa_p$$ the pullback of the function $g$ to $\{x^{m_1}=y^{m_2}z^{m_3}\}\cap \Upk$. Up to a unit, $$\gk=x^{n_1}.$$

Under the morphism $\pi_\pt$, the preimage of $$\Courbeskt\cap \Ut=norm^* \left(\left\{yz=x=0\right\}\right)\cap \Ut$$ is a chain of complex curves, namely $$Str\left(\frac{m_1}{d};\frac{m_2}{d},\frac{m_3}{d}|n_1;0,0\right)$$ where the notation $Str$ has been introduced in \ref{appendix resol HJ}. The multiplicities on the curves of the string are those of the function $$\gti:=g\circ \Pi$$ the pullback of the function $g$ to $\Sb\cap \pi^{-1}(\Ut)$, and where $\widetilde{C_0}=\Cpti$, and $\widetilde{C_{l+1}}=\Cti$.

\noindent Note that the morphism $\pi_\pt \colon \Uti \rightarrow \Ut$ is an isomorphism outside of $\pt$.

\begin{discussion}{\bf Orientation compatibilities, point $\minus$.}

Let us remind here the choices of orientation made in subsection \ref{orientation St}. 

The pullbacks of the orientations of the curves $\Ct\cap \Ut,\Cpt\cap \Ut$ by the biholomorphism $\pi_\pt$ are again the opposites of the orientations given by their local complex structure. 

Each new curve $\widetilde{C_i}$ is taken oriented by \textbf{the opposite of its complex structure}. 

The pullback of the orientation of $\St\cap \Ut$ by the bimeromorphism $\pi_\pt\colon \Uti \rightarrow \Ut$ is, again, the opposite of the orientation given by the complex structure. Now, the compatibility of the complex orientations implies that at each intersection point $\widetilde{C_i}\cap \widetilde{C_{i+1}}$, the combination of the orientations of the two curves gives the opposite of the ambient orientation of $\Uti$.
\end{discussion}

\begin{figure}[h]
\begin{center}

        \includegraphics[totalheight=1.1cm]{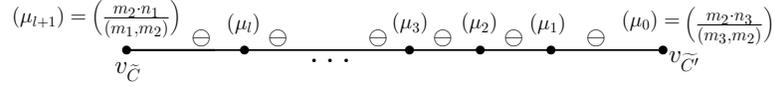}
    \caption{The string $Str^\minus\left(\frac{m_1}{d};\frac{m_2}{d},\frac{m_3}{d}|n_1;0,0\right)$.}
    \label{fig:chainemoins}
\end{center}
\end{figure}

Recall that $(a,b)$ denotes $gcd(a,b)$.

In the end, the preimage of $(\Courbeskt,\pt)$ by $\pi_\pt$ is a chain of complex curves, whose dual decorated graph is the bamboo of figure \ref{fig:chainemoins}, where the $\minus$ signs refer to the orientation compatibilities. The integers $\mu_i$ represent the multiplicities of $\gti$ on each curve. Again, for the definition of the multiplicity of $\gti$ on the real surfaces $\Cti$ and $\Cpti$, see Definition \ref{def multiplicity}.

\subsection{The decorated graph $\Gmult$}
The nature of the resolution process described in Appendix \ref{appendix resol HJ}, affecting each coordinate axis in only one point, whose preimage is a point, implies:
\begin{lemma}\label{fact euler char}
If $\Ct$ is an irreducible surface of $\Courbeskt$, and $\Cti$ is the strict transform of $\Ct$ by $\pi$, then \begin{equation}\label{eq euler char}
\chi(\Cti)=\chi(\Ct).
\end{equation} Furthermore, the surfaces added in the resolution process are all of genus $0$. 
\end{lemma}

\begin{defn}
Denote $\Gmult$ the graph obtained from $\Gkdt$ by introducing the strings described in Subsections \ref{subsection resol plus} and \ref{subsection resol moins}, where 
\begin{itemize}
\item Vertices coming from arrowheads are still represented as arrowheads, and every vertex is decorated with the corresponding multiplicity of $\gti$.
\item The vertex $v_{\Cti}$ corresponding to the strict transform of the curve $\Ct$ by $\pi$ has genus decoration \begin{equation}\label{eq final genus}
\left[g_{\Cti}\right]=[1-\frac{\chi(\Ct)}{2}].
\end{equation}
and the new vertices introduced in the strings have genus decoration $0$.
\end{itemize}
\end{defn}

The genus decorations of Equation (\ref{eq final genus}) are motivated by the Euler characteristic identities of Equation (\ref{eq euler char}). For the computation of $\chi\left(\overline{C}\right)$, see Equation (\ref{eq calcul carac euler}) of Theorem \ref{thm nombre courbes normalisation}.

\section{The boundary of the Milnor fiber}\label{section bdry Milnor fiber}
Now, one has all the information required to compute the boundary of the Milnor fiber of $f$. Recall that this one is identified \emph{via} Proposition \ref{prop diffeomorphic boundaries} to $\partial \Sk=\rho_{|\Sk}^{-1}(\varepsilon)$, for any $\varepsilon >0$ small enough in the sense of Definition \ref{def eps small}. 

Denote $\widetilde{\rho}:=\rho\circ \Pi$. The fact that $\Pi$ is an orientation-preserving real analytic isomorphism outside $0$ allows us to say that $\partial \Sk$ is orientation-preserving homeomorphic to its preimage by $\Pi$, which is the boundary $\partial T$ of the tubular neighbourhood $$T:=\{\widetilde{\rho}\leqslant \varepsilon\}$$ of the preimage $E$ of the origin by $\Pi$.

Now, $\gti$ and $\widetilde{\rho}$ are respectively an adapted and a rug function for the configuration $E$ in $\Sb$. Theorem \ref{thm corresp boundary plumbing graph} implies that $$\partial T \simeq M_{\Gdp{E}{\Sb}},$$ where the equivalence symbol denotes an orientation-preserving homeomorphism.

In conclusion, the orientation-preserving diffeomorphism between $\partial \Sk$ and $\partial T$ implies $$\partial \Sk \simeq M_{\Gdp{E}{\Sb}},$$ where the equivalence symbol denotes an orientation-preserving diffeomorphism.

\begin{defn}
Denote $\Gresoldc$ the graph obtained from $\Gmult$ by keeping the genera decorations and replacing the multiplicity decorations of nodal vertices by the self-intersection decorations, using Lemma \ref{lemma compute self-intersections}, then removing the arrowhead vertices.
\end{defn}
The construction of the decorations of $\Gresoldc$ is made in order to have the equality $\Gresoldc=\Gdp{E}{\Sb}$.

Finally we get the following generalization of \cite[Theorem 10.2.10]{NemSzi12}:

\begin{thm}\label{thm main}
The boundary of $F$ is a graph manifold, and a possible plumbing graph for $\partial F$ is the decorated graph $\Gresoldc$, which has only nonnegative genera decorations.
\end{thm}

This implies in particular our Theorem of the Introduction.
%
%
%
%

\begin{appendices}

\section{Tubular neighbourhoods and graph manifolds}
Graph manifolds can appear naturally as boundaries of neigbourhoods of complex curves in smooth complex surfaces, see \cite{Mum61}. In our context, they will appear as boundaries of neighbourhoods of real surfaces inside a smooth $4$-manifold.

\begin{defn}{\bf (Simple configuration of surfaces, and its plumbing dual graphs.)}\label{def simple conf of surfaces, plumbing graphs.}

Let $\St$ be a $4$-dimensional oriented real analytic manifold. A \textbf{simple configuration of compact real analytic surfaces in $\St$} is a subset $E\subset \St$ such that:
\begin{enumerate}
\item $E=\bigcup\limits_{finite}E_i$, such that each $E_i$ is an oriented closed smooth real analytic surface.
\item For all $i\neq j \neq k \neq i$, the intersection $E_i\cap E_j \cap E_k$ is empty.
\item For all $i\neq j$, the intersection $E_i \cap E_j$ is either empty or transverse. In particular, it is a finite union of points.
\end{enumerate}
In this setting, one defines a \textbf{plumbing dual graph} $\Gdp{E}{\St}$ of $E$ in $\St$ by decorating its dual graph in the following way:
\begin{enumerate}
\item Decorate each vertex $v_{E_i}$ by the self-intersection $e_i$ of $E_i$ in $\St$ and by the genus $[g_i]$ of the surface $E_i$. 
\item Decorate each edge of $\Gra[E]$ corresponding to the intersection point $p$ of $E_i\cap E_j$, by $\oplus$ if the orientation of $E_i$ followed by the orientation of $E_j$ is equal to the orientation of $\St$ at $p$, and by $\ominus$ otherwise. 
\end{enumerate}
\end{defn}

\begin{defn}\label{def rug function}(See \cite{Dur83}.)

Let $E$ be a simple configuration of compact orientable real analytic surfaces in an oriented $4$-dimensional real analytic manifold $\St$. In this context, we call \textbf{rug function} any real analytic proper function $\rho\colon \St \rightarrow \R_+$ such that $\rho^{-1}(0)=E$.
\end{defn}

\begin{theorem}\label{thm corresp boundary plumbing graph}
Let $E$ be a simple configuration of compact orientable real analytic surfaces in an oriented $4$-dimensional real analytic manifold $\St$, that admits a rug function $\rho$ as in Definition \ref{def rug function}. Then, for $\varepsilon >0$ small enough, the boundary of the oriented $4$-manifold $\{\rho\leqslant\varepsilon\}$ is orientation-preserving homeomorphic to the graph manifold associated to the graph $\Gdp{E}{\St}$.
\end{theorem}

\begin{defn}
In this setting, the manifold $\{\rho\leqslant\varepsilon\}$ is called a \textbf{tubular neighbourhood} of $E$. 
\end{defn}

\begin{proof}[About the proof of Theorem \ref{thm corresp boundary plumbing graph}.]

This theorem can be seen as an extension of what is done in \cite{Mum61}, in the case of a configuration of complex analytic curves in a smooth complex surface. 

In our case, observe first that we can extend the definition of rug functions to semi-analytic functions $\rho$, and still have a unique homeomorphism type for the boundary of the neighbourhood $\{\rho\leqslant\varepsilon\}$ of $E$ for $\varepsilon>0$ small enough, following the proof of \cite[Proposition 3.5]{Dur83}. Now, one can build by hand a semi-analytic neighbourhood whose boundary is homeomorphic to the manifold $\Gdp{E}{\St}$. 

This is done by building a rug function for each irreducible component $E_i$ of $E$, providing a tubular neighbourhood $T_i$ of each $E_i$ whose boundary is an $\Sp^1$-bundle of Euler class $e_i$ over $E_i$. One then plumbs those bundles using appropriate normalizations of the rug functions, building a semi-analytic neighbourhood of $E$ which is homeomorphic to the desired graph manifold.
\end{proof}

\begin{rk}
Note that the decorations on the edges of the graph $\Gdp{E}{\St}$ depend on the orientations of the surfaces $E_i$. However, if the surfaces $E_i$ are only orientable, the different possible plumbing dual graphs still encode the same graph manifold, see move {\bf [R0]} of the plumbing calculus in \cite{Neu81}.
\end{rk}

We end this section with a tool that can be used to compute the self-intersections of the irreducible components of $E$:

\begin{defn}\label{def adapted function}\index{Adapted function}
Let $E=\bigcup\limits_{finite}E_i$ be a simple configuration of compact oriented real analytic surfaces in a $4$-dimensional real analytic manifold $\St$. A real analytic function $g\colon \St \rightarrow \C$ is called \textbf{adapted} to $E$ if 
\begin{enumerate}
\item $E_{tot}:=g^{-1}(0)$ is a simple configuration of oriented, not necessarily compact, real analytic surfaces, such that $E_{tot}\supset E$.
\item \begin{enumerate}
\item \label{mutiplicity gen} For any component $E_i$ of $E_{tot}$, $\forall ~ p\in E_i \setminus \bigcup\limits_{j \neq i}E_j$, there is a neighbourhood $U_p$ of $p$ in $\St$ and complex coordinates $(x_p,y_p)$ on $U_p$ such that $U_p\cap E_i=\{x_p=0\}$ and $n_i\in \N^*$ such that $$g=x_p^{n_i} \cdot \varphi$$ where $\varphi\colon U_p \rightarrow \C$ is a unit at $p$.

\item \label{multiplicity double} For any components $E_i$ of $E$, $E_j$ of $E_{tot}$, $\forall ~ P\in E_i\cap E_j$, there is a neighbourhood $U_p$ of $p$ in $\St$ and complex coordinates $(x_p,y_p)$ on $U_p$ such that $U_p\cap E_i=\{x_p=0\}$, $U_p\cap E_k=\{y_p=0\}$,  and $n_i, n_k\in \N^*$ such that $$g=x_p^{n_i}y_p^{n_j} \cdot \varphi$$ where $\varphi\colon U_p \rightarrow \C$ is a unit at $p$.
\end{enumerate}
\end{enumerate}
\end{defn}

\begin{defn}\label{def multiplicity}
In this setting, the integer $n_i$ of point \ref{mutiplicity gen} of Definition \ref{def adapted function}, independent of the point $p\in E_i \setminus \bigcup\limits_{j \neq i}E_j$, is called \textbf{the multiplicity of $g$ on $E_i$}, denoted $m_{E_i}(g)$.
\end{defn}

\begin{lemma}{\bf (Computing self-intersections.)}\label{lemma compute self-intersections}

Let $\St,E,g,E_{tot}$ be as in Definition \ref{def adapted function}. Let $E^{(1)}$ be an irreducible component of $E$. Then the self-intersection $e^{(1)}$ of the surface $E^{(1)}$ in $\St$ verifies the following condition: let $p_1,\cdots,p_n$ be the intersection points of $E^{(1)}$ with other components of $E_{tot}$, $p_j\in E_j\cap E^{(1)}$, where the same component may appear several times. Then $$n^{(1)}\cdot e^{(1)}+\sum\limits_{i=1}^{n} \epsilon_i \cdot n_i=0$$ where $\epsilon_i\in \{-1,+1\}$ refers to the sign associated to the intersection $p_i$ in the following sense: if $p\in E_i\cap E_j$, associate $+1$ to $p$ if and only if the combination of the orientations of $E_i$ and $E_j$ at $p$ provides the ambient orientation of $\St$.
\end{lemma}
\begin{proof}
The proof follows the standard argument in the holomorphic category. The difference of the two members of the equation is the intersection number of $E^{(1)}$ with the cycle defined by $g=0$. This cycle is homologous with that defined by a nearby level of $g$, which does not meet $E^{(1)}$ any more. The intersection number being invariant by homology, one gets the desired result. 

In order to make this argument rigorous, one has to work in convenient tubular neighborhoods of $E$ and to look at the cycles defined by the levels of $g$ in the homology of the tube relative to the boundary.
\end{proof}
\section{Graph coverings}\label{subs graph cover}

In this subsection we introduce the notion of graph covering, which will be useful when describing the method of computation of a plumbing graph for the manifold $\partial F$. Here, $\Gamma$ is a non-oriented graph, and $\VG,\EG$ denote respectively the set of vertices and the set of edges of $\Gamma$.

\begin{defn}
\index{$\Z$-action on a graph} (See \cite[1.3]{Nem00}.) We say that $\Z$ \textbf{acts on a graph $\Gamma$} if there are group actions $$a_\PV\colon \Z\times\VG\rightarrow \VG$$ and $$a_\PE\colon \Z\times\EG\rightarrow \EG$$ such that if $v_1$ and $v_2$ are the end-vertices of $e\in \EG$, then $a_\PV(1,v_1)$ and $a_\PV(1,v_2)$ are the end-vertices of $a_\PE(1,e)$.

\end{defn}

\begin{defn}\label{def ordre cyclique}
Let $F$ be a finite set. A \textbf{cyclic order} on $F$ is an automorphism of $F$ that generates a transitive group of automorphisms. That is, as a permutation it is a cycle.
\end{defn}

\begin{ex}
A $\Z$-action on a graph $\Gamma$ induces a cyclic order on each orbit, made either of edges or vertices, under this action. Reciprocally, the data of cyclic orders on sets partitioning $\VG$ and $\EG$, with the appropriate compatibility axioms, gives rise to a $\Z$-action on $\Gamma$.
\end{ex}

\begin{defn}
Let $\Gamma$ be a graph, endowed with the trivial action of $\Z$. A \textbf{$\Z$-covering}, or \textbf{cyclic covering}, of $\Gamma$, is a graph $G$ such that $\Z$ acts on $G$, together with an equivariant map $p \colon G\rightarrow \Gamma$, such that the restriction of the action of $\Z$ to any set of the type $p^{-1}(v)$ or $p^{-1}(e)$ is transitive.

In this setting, if $v\in \VG$ and $e\in \EG$, then the elements of $p^{-1}(v)$ are called the \textbf{vertices associated to $v$}, while the elements of $p^{-1}(e)$ are called the \textbf{edges associated to $e$}.
\end{defn}

\begin{rk}
Let $G$ be a cyclic covering of a graph $\Gamma$. For any $v\in \VG$ or $e\in \EG$, denote $n_v$, respectively $n_e$, the cardinal of $p^{-1}(v)$, respectively $p^{-1}(e)$. Then if $v_1,v_2\in \VG$ are the end-vertices of $e\in \EG$, there is $d_e\in \N^*$ such that $$n_e=d_e\cdot lcm(n_{v_1},n_{v_2}).$$
\end{rk}

\begin{defn}
A \textbf{covering data} for a graph $\Gamma$ is a system of positive integers $({\bf n_v},{\bf n_e})=\left\{\{n_v\}_{v\in \VG},\{n_e\}_{e\in \EG}\right\}$ such that for all $e\in \EG$ with end-vertices $v_1,v_2$, the number $n_e$ is a multiple of $lcm(n_{v_1},n_{v_2})$.
\end{defn}

\section{Hirzebruch-Jung chains}\label{appendix resol HJ}

We introduce here the Hirzebruch-Jung chains (or strings) that occur in the final step of our resolution of $(\Skt,0)$. For more details, one can consult \cite[III.5]{BPV84}, \cite[4.3]{NemSzi12}.

Let $a,b,c\in \N$. Recall that $(a,b,c)$ denotes $gcd(a,b,c)$. Suppose that $d:=(a,b,c)=1$. In these conditions, $(V,p):=\left(\{x^a=y^bz^c\},0\right)^N$ is an isolated singularity of complex surface. We describe here a graph of resolution of this singularity, decorated with the self-intersections of the irreducible components of the exceptional divisor and the multiplicities of the pullback of the function $g(x,y,z)=x^{n_1}y^{n_2}z^{n_3}$.

Denote $\delta:=\dfrac{a}{(a,b)(a,c)}$, and by $\alpha$ the unique integer in $[0,\delta-1]$ such that $$a~|~\alpha c (a,b)+b(a,c).$$

Let $$\frac{\delta}{\alpha}=k_1-\cfrac{1}{k_2-\cfrac{1}{\cdots-\cfrac{1}{k_l}}}$$ be the negative fraction expansion of $\delta / \alpha$, $k_i\geqslant 2$, $k_i\in \N$.

Define $\mu_{l+1}:=\dfrac{b\cdot n_1+a\cdot n_2}{(a,b)},\mu_{0}:=\dfrac{c\cdot n_1+a\cdot n_3}{(a,c)},\mu_1:=\dfrac{\alpha\cdot \mu_0+\mu_{l+1}}{\delta}$ and let $\mu_2,\cdots,\mu_l$ be defined by the relation $$\forall ~ 1\leqslant i \leqslant l, \mu_{i+1}=k_i\cdot\mu_i-\mu_{i-1}.$$

Then the graph of Figure \ref{fig:HJbambouxyz} is a graph of resolution of $(V,0)$. Left and right-hand arrows represent, respectively, the pullbacks of the $z$ and the $y$-axes. The numbers between parentheses are the multplicities of the pullback of the function $g(x,y,z)=x^{n_1}y^{n_2}z^{n_3}$ on the irreducible components of its total transform.

\begin{figure}[h]
\begin{center}

        \includegraphics[totalheight=1.4cm]{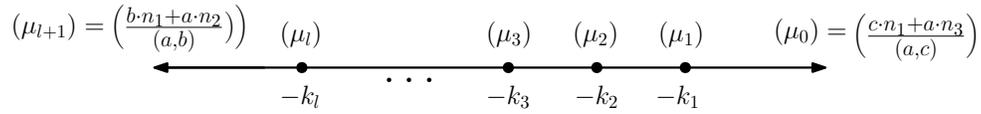}
    \caption{The string $Str(a;b,c|n_1;n_2,n_3)$.}
    \label{fig:HJbambouxyz}
\end{center}
\end{figure}

We denote by $Str^\oplus(a;b,c|n_1;n_2,n_3)$, resp. $Str^\ominus(a;b,c|n_1;n_2,n_3)$ the chain of Figure \ref{fig:HJbambouxyz} with each edge decorated by $\plus$, resp. $\minus$, and the self-intersection decorations removed.
\end{appendices}
\bibliographystyle{alpha}
\bibliography{biblio} 

\Adresses
\end{document}